\documentclass[sn-mathphys-num]{sn-jnl}%

\usepackage{amsmath,amssymb,amsthm}
\usepackage{multirow}
\usepackage{graphicx}
\usepackage{longtable}
\usepackage{gensymb}
\usepackage{subcaption}
\usepackage{algpseudocode}
\usepackage[ruled,vlined]{algorithm2e}
\usepackage{cancel}
\usepackage{colortbl}
\usepackage{caption}
\usepackage{subcaption}
\usepackage{array}
\usepackage{tikz}
\usepackage{tikz-cd}
\usepackage{diagbox}

\usepackage{hyperref}
\usepackage{cleveref}

\newcommand{\blue}{}

\newcommand{\rd}{{\mathrm{d}}}

\newtheorem{theorem}{Theorem}%
\newtheorem{prop}[theorem]{Proposition}%
\newtheorem{lemma}[theorem]{Lemma}%
\newtheorem{cor}[theorem]{Corollary}%

\newtheorem{remark}{Remark}%

\newtheorem{definition}{Definition}%

\begin{document}

\title[Measure-Theoretic Time-Delay Embedding]{Measure-Theoretic Time-Delay Embedding}

\author[1]{\fnm{Jonah} \sur{Botvinick-Greenhouse}}\email{jrb482@cornell.edu}
\equalcont{These authors contributed equally to this work.}

\author[1]{\fnm{Maria} \sur{Oprea}}\email{mao237@cornell.edu}
\equalcont{These authors contributed equally to this work.}

\author[2]{\fnm{Romit} \sur{Maulik}}\email{rmaulik@psu.edu}

\author*[3]{\fnm{Yunan} \sur{Yang}}\email{yunan.yang@cornell.edu}

\affil[1]{\orgdiv{Center for Applied Mathematics}, \orgname{Cornell University}, \orgaddress{\street{657 Frank H.T.~Rhodes Hall}, \city{Ithaca}, \postcode{14850}, \state{NY}, \country{USA}}}

\affil[2]{\orgdiv{College of Information Sciences and Technology}, \orgname{Pennsylvania State University}, \orgaddress{\street{E397 Westgate Building}, \city{University Park}, \postcode{16802}, \state{PA}, \country{USA}}}

\affil*[3]{\orgdiv{Department of Mathematics}, \orgname{Cornell University}, \orgaddress{\street{310 Malott Hall}, \city{Ithaca}, \postcode{14850}, \state{NY}, \country{USA}}}

\abstract{
The celebrated Takens' embedding theorem provides a theoretical foundation for reconstructing the full state of a dynamical system from partial observations. However, the classical theorem assumes that the underlying system is deterministic and that observations are noise-free, limiting its applicability in real-world scenarios. Motivated by these limitations, we formulate a measure-theoretic generalization that adopts an Eulerian description of the dynamics and recasts the embedding as a pushforward map between spaces of probability measures. Our mathematical results leverage recent advances in optimal transport. Building on the proposed measure-theoretic time-delay embedding theory, we develop a computational procedure that aims to reconstruct the full state of a dynamical system from time-lagged partial observations, engineered with robustness to handle sparse and noisy data. We evaluate our measure-based approach across several numerical examples, ranging from the classic Lorenz-63 system to real-world applications such as NOAA sea surface temperature reconstruction and ERA5 wind field reconstruction.
}


\maketitle

\section{Introduction}
Dynamical systems provide a universal language for modeling the temporal evolution of complex systems, appearing across a diverse range of scientific disciplines, including physics, biology, chemistry, ecology, and social sciences. A dynamical system comprises of a space (denoted by $M$) defining the possible states ($x$) of the system and a rule describing the evolution of these states over time ($t$). Understanding the behavior of complex dynamics allows for accurate predictions of future states based on historical data, which is crucial in fields such as weather prediction, financial market analysis, and epidemiology \citep{schneider2010water,ghil2020physics,lyneis2000system,ye2015distinguishing,botvinick2023learning,anderson2004epidemiology,khyar2020global}.

In practice, the exact equations governing a system's behavior are often unknown, and one may have to study the system's evolution through empirically collected time series data. This indirect interaction with the system's full dynamics is frequently complicated by experimental limitations preventing complete measurement of the full state. For instance, only the first coordinate of a $d$-dimensional state $x$ may be observed. Such partial and potentially limited observational data makes it challenging to accurately model the system's dynamics.

In situations where the full state is not directly observable, time-delay embedding has become a fundamental technique in the analysis of dynamical systems. This method involves concatenating time-lagged versions of a scalar time-dependent observation into a high-dimensional state vector, resulting in a system that is topologically equivalent to the full, unobserved dynamics.  The celebrated Takens' embedding theorem supplies the theoretical foundation for time-delay embedding and has inspired a myriad of works over the last four decades that perform data-driven analysis on partially observed nonlinear systems~\citep{Takens1981DetectingSA, kim1999nonlinear}. Notable applications of time-delay embedding include forecasting~\citep{bakarji2022discovering, young2023deep}, noise reduction~\citep{grassberger1993noise, kirtland2023unstructured}, control~\citep{nitsche1992controlling}, as well as various biological studies~\citep{ye2015distinguishing, raut2023arousal, smith2011introduction}. The theory of time-delay embeddings has been extended beyond classical reconstruction to address problems of prediction and modeling of dynamical evolution~\cite{takens2002reconstruction}. It has been shown that reliable short-term forecasts can be achieved even when embeddings contain self-intersections~\cite{schroer1998predicting}, while subsequent work established a rigorous mathematical framework for prediction from such non-injective embeddings~\cite{baranski2024prediction}. More recently, practical tests for assessing topological conjugacy directly from time series have been proposed~\cite{dlotko2024testing}.

However, Takens' embedding theorem assumes that the underlying system follows precise, predictable rules without random perturbations and that observables are measured perfectly without errors. In practical scenarios, both the underlying system and the observables are subject to intrinsic and extrinsic noise. Intrinsic noise refers to the inherent unpredictability within the system, such as thermal fluctuations or quantum effects, while extrinsic noise includes measurement errors, environmental disturbances, and observational limitations--external factors that can affect the data. These sources of randomness challenge the idealized assumptions of Takens' theorem when modeling partially observed dynamical systems.

Given these practical challenges, it is necessary to consider variants of time-delay embedding that incorporate assumptions of randomness. Notably, it was shown in \citep{sauer1991embedology}  that the time-delay map is an embedding almost surely in the space of observation functions, and in \citep{probabilistic1} that the delay map formed by a polynomially perturbed observation function is an embedding at almost all points. However, these works do not address intrinsic or extrinsic noise in the state itself. On the other hand, \cite{casdagli1991state} leveraged statistical techniques to quantify the effect of i.i.d.~extrinsic noise on state reconstruction and time-series prediction. Moreover, the works \cite{stoch_takens, noise2} focus on the embedding properties of individual Lagrangian trajectories subject to randomness and forcing.

\begin{figure}
\centering
\subfloat[Pointwise time-delay embedding]{ \includegraphics[width = .8\linewidth]{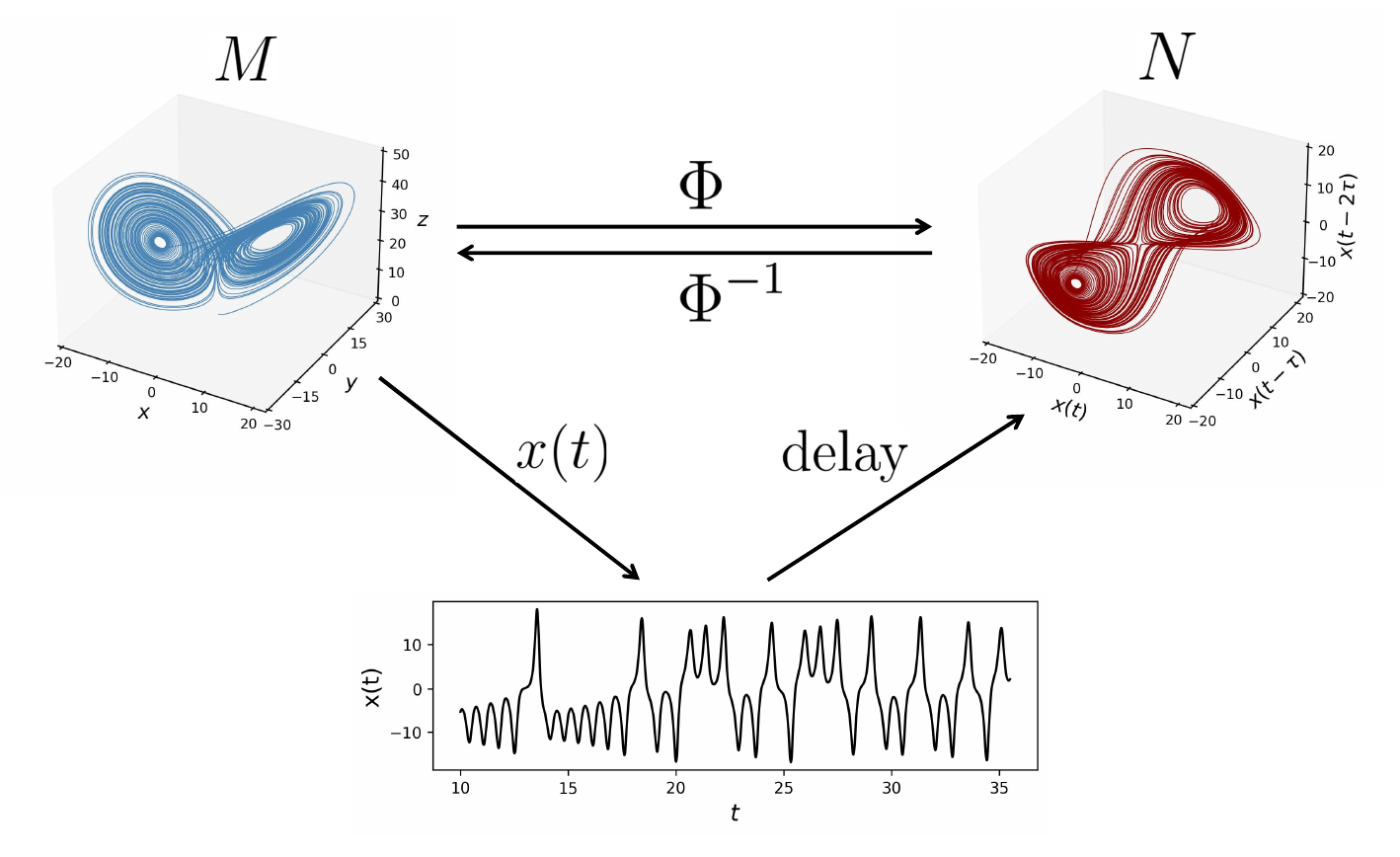}}\label{fig:point_embedding}\\
\subfloat[Measure-theoretic time-delay embedding]{ \includegraphics[width = .95\linewidth]{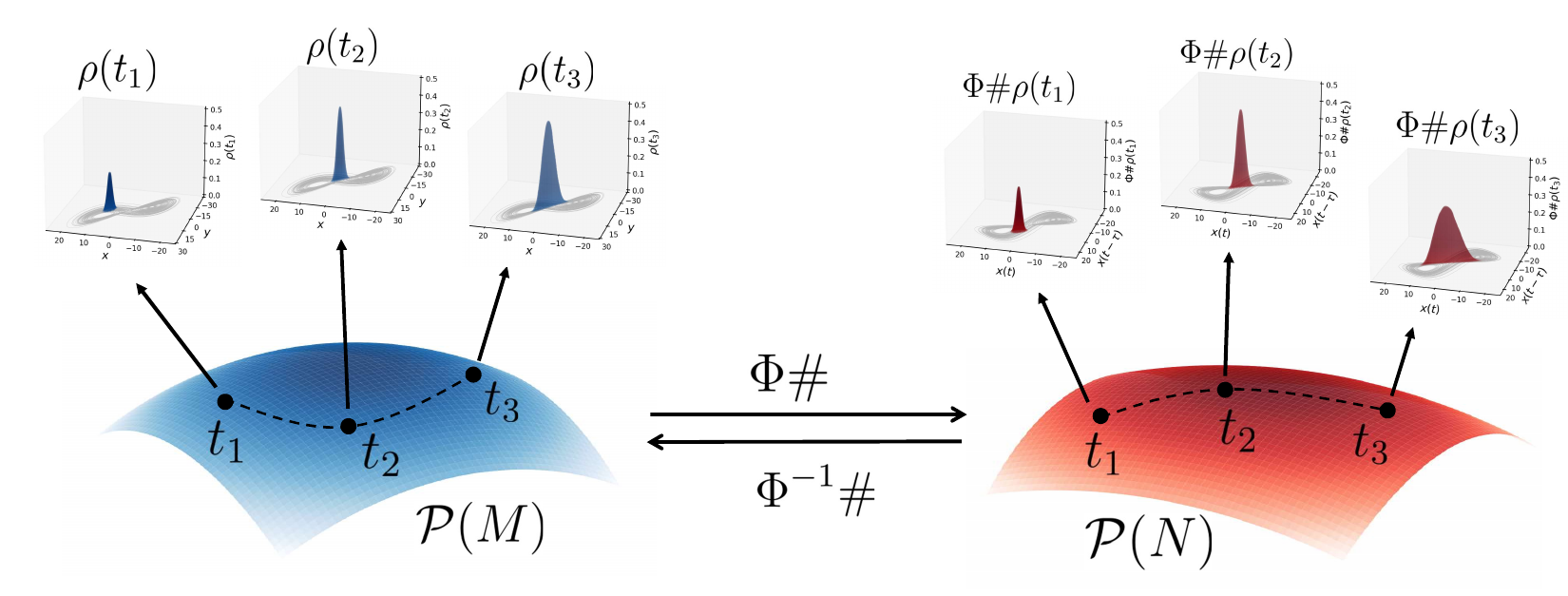}}
    \caption{Illustration of the differences between the classical pointwise time-delay embedding (a) and the proposed measure-theoretic time-delay embedding (b).\label{fig:main}}
    \label{fig:prob_embedding}
\end{figure}

In this work, we take a different approach by \textbf{lifting} both the domain ($M$) and the co-domain ($N$) of an embedding map $\Phi:M\to N$ to the space of probability measures over $M$ and $N$, respectively, denoted by $\mathcal{P}(M)$ and $\mathcal{P}(N)$. These two infinite-dimensional probability spaces are connected by pushforward maps acting on probability measures (see~Fig.~\ref{fig:main}). One main contribution of our work is to rigorously study the embedding property between $\mathcal{P}(M)$ and $\mathcal{P}(N)$ under an Eulerian description of the dynamics. We establish the existence of a \textbf{smooth, one-to-one, and structure-preserving} map that translates fluid/gas flows represented as probability distributions over the state coordinates into their counterparts characterized in the time-delay coordinates. Moreover, this probabilistic embedding map is precisely the pushforward action of the original Lagrangian embedding map $\Phi$. The theory of optimal transport~\citep{vilani_topics} plays a critical role in the analysis, supplying essential mathematical tools such as the differentiability of maps between probability spaces and tangent space structures.

Our second main contribution is leveraging the measure-theoretic time-delay embedding theory to establish a robust computational framework for learning the inverse embedding function ($\Phi^{-1}$) from data. This function, known as the full-state reconstruction map, typically has no analytical form but is crucial for {reconstructing} the complete state of a nonlinear system from partially observed data. To learn the reconstruction map, we first extract empirical measures in both the full-state space ($M$) and the delay space ($N$) based on a single time trajectory. We then select a suitable metric over the probability space as the objective function and enforce that the pushforwards of the empirical measures in the delay space match the corresponding measures in the full-state space. In contrast with our approach, existing methods commonly learn the reconstruction map using the pointwise mean-squared error loss~\citep{bakarji2022discovering, raut2023arousal, young2023deep}. When the observed data is sparse and noisy, the accuracy of such pointwise methods can be compromised, while our measure-theoretic approach {demonstrates better robustness}. In Section~\ref{sec:experiments}, we showcase the effectiveness of our proposed approach through various numerical tests on classic synthetic examples, such as the Lorenz-63 and Lotka--Volterra systems, as well as large-scale real-world applications, including the reconstruction of NOAA Sea Surface Temperature~\citep{AnImprovedInSituandSatelliteSSTAnalysisforClimate} and ERA5 wind speed datasets from partial observations~\citep{https://doi.org/10.1002/qj.3803}.

{There have been recent advances in leveraging delay-coordinate liftings to apply linear tools to nonlinear dynamics, including kernel integral operators that concentrate the Koopman point spectrum, the forced-linear HAVOK decomposition, and a Koopman--Takens theorem for linear least-squares prediction in the infinite-delay limit \cite{das2019delay,brunton2017chaos,koltai2024koopman}. In parallel, measure-theoretic approaches employ optimal transport to construct transfer operators and develop reduced-order models via displacement interpolation and augmentation \cite{koltai2021transfer,khamlich2025optimal}. Our framework links these efforts by \emph{lifting delay embedding at the functional level}: we pass from trajectories to probability measures, model evolution as pushforwards on the probability space, and obtain a smooth, one-to-one distributional embedding. This Eulerian lifting complements Koopman/transfer-operator analyses and yields stable reconstruction and forecasting frameworks under finite, noisy, and sparsely sampled time-lagged observations.
}

The rest of the paper is organized as follows. Section~\ref{sec:background} reviews the essential background on Takens' theorem and optimal transport. In Section \ref{sec:measure_embedding}, we present our main theoretical result, which generalizes Takens' theorem to the space of probability distributions. In Section~\ref{sec:experiments}, we introduce our computational framework for learning the inverse embedding map from data and demonstrate its robustness across several synthetic and real data examples. Conclusions follow in Section~\ref{sec:conclusions}.

\section{Background and Overview}\label{sec:background}

Essential mathematical notations 
are summarized in Table \ref{tabel:notation}.

\subsection{Takens' Embedding Theorem}
Many physical processes can be modeled by systems of ordinary differential equations (ODEs), which can be represented in the form $\dot{x} = v(x)$, where $x \in M$ is the state, $M$ is a smooth compact $d$-dimensional manifold and $v$ is a $\mathcal{C}^2$ vector field on $M$.  Given an initial condition $x\in M$, we denote the solution to the ODE by $\{\phi_t(x)\}_{t\geq 0}$, which is often referred to as the trajectory. Moreover, $\phi_t:M\to M$ is known as the time-$t$ flow map. The trajectory provides critical information about the underlying dynamical system and is useful in various engineering and data science applications, such as parameter identification and model reduction \citep{benner2015survey,brunton2016discovering}. However, when the state dimension $d$ is large, it is often the case that experimentalists are unable to directly access the trajectory $\{\phi_t(x)\}_{t\geq 0}$, but instead have access to time-series projections of the form $\{h(\phi_t(x))\}_{t \geq 0},$ where $h\in \mathcal{C}^2(M,\mathbb{R})$ is an observation function. Thus, it is essential to understand to what extent the time series $\{h(\phi_t(x))\}_{t  \geq 0}$ can provide information on the full trajectory $\{\phi_t(x)\}_{t \geq 0}$. 

Takens' embedding theorem (1981) establishes criteria under which the partial observations of a dynamical system can be used to reconstruct the full state. This reconstruction is possible because the original trajectory $\{\phi_t(x)\}_{t \geq 0}$ and the time-lags of the partially observed trajectory $\{h(\phi_t(x))\}_{t \geq 0}$ are related through a structure-preserving diffeomorphism, known as an embedding. A precise definition is given in Definition \ref{def:embeddin} below.
\begin{definition}[Embedding]\label{def:embeddin}
    Let $M, N$ be two differentiable manifolds. Then a function $f:M \to N$ is an embedding if $f$ is a diffeomorphism such that the derivative operator $df_x:T_xM \to T_{f(x)}N$ is injective $\forall x \in M$.
\end{definition}
The statement of Takens' theorem is given in Theorem~\ref{thm:classic_takens}.

\begin{theorem}[Takens' Embedding Theorem]\label{thm:classic_takens}
    Let $\tau > 0$, choose $m \geq 2d+1$, and suppose that $v$ satisfies the following:
    \begin{enumerate}
        \item[(i)] If $v(x) = 0$, then the {eigenvalues of $(d\phi_\tau)_x:T_xM \rightarrow T_xM$ are all different, and none of them equals $1$.}
        \item[(ii)] No periodic integral curve of $v$ has period equal to $k\tau$ for $k \in \{1, \dots, m\}$.
    \end{enumerate}
    Then, it is a generic property that the delay coordinate map given by
\begin{equation}\label{eq:delay}
    \Phi_{h, \phi_{\tau}}(x):= \left(h(x),h(\phi_{\tau}(x)),\dots ,h\left(\phi_{(m-1)\tau}(x)\right) \right)  \in \mathbb{R}^m\,,
\end{equation}
is an embedding.
\end{theorem}

\begin{table}
    \centering
\begin{tabular}{ll}
\hline
    Notation & Meaning\\ \hline
    $M$ & Compact $d$ dimensional manifold \\ \hline
    $\phi_t$ & The flow map of a dynamical system on $M$ \\ \hline
    $h$ & The observation function $h: M \rightarrow \mathbb{R}$\\ \hline
    $\tau$ & The delay parameter belonging to $\mathbb{R}_{> 0}$ \\ \hline
   $m$ & Dimension of the delay embedding space\\ \hline
   $\rho$, $\rho_0$ &  Measures on $M$ \\ \hline
   $\Phi_{h, \phi_{\tau}}$ & Takens' delay embedding map \\ \hline
   $\Psi_{h, \phi_{\tau}}$ & The delay embedding map for distributions \\ \hline
   $\rho_t$ & Curve of measures starting at  $\rho_0 = \rho$ \\ \hline
   $|\rho_t'|$ & Metric derivative of $\rho_t$ according to Definition \ref{def:metric_derivative}\\ \hline
   $\frac{d}{dt}\rho_t = v_t$ & Tangent vector field along the curve $\rho_t$ as in Definition \ref{def:vector_field}\\ \hline
   $\mathcal{F}(X, Y)$ & The space of measurable functions from $X$ taking values in $Y$\\ 
   \hline
\end{tabular}
\caption{A list of mathematical notations in Sections~\ref{sec:background} and~\ref{sec:measure_embedding}.}\label{tabel:notation}
\end{table}

In \eqref{eq:delay}, $m\in \mathbb{N}$ is known as the embedding dimension, and $\tau > 0$ is the so-called time-delay. By ``generic'' we mean that the collection of observation functions $h\in \mathcal{C}^2(M,\mathbb{R})$ for which~\eqref{eq:delay} defines an embedding is open and dense in the $\mathcal{C}^2$ topology. In particular, when \eqref{eq:delay} is an embedding, the delayed trajectory $\{\Phi_{h, \phi_{\tau}}(\phi_t(x))\}_{t\geq 0}$ is topologically equivalent to the original orbit $\{\phi_t(x)\}_{t\geq 0}\subseteq M$. %
This perspective has motivated the use of delay coordinates in various applications, such as time-series prediction~\citep{farmer1987predicting}, attractor reconstruction~\citep{pecora2007unified}, causality detection~\citep{sugihara2012detecting}, and noise reduction~\citep{grassberger1993noise,kirtland2023unstructured}.

Often, dynamical systems asymptotically approach a compact attractor $A$ with fractal dimension $d_A \ll d$. Ideally, the embedding dimension should depend on $d_A$ rather than $d$, in order to ensure that the attractor $A$ is embedded using the delay map. In~\citep{sauer1991embedology}, Takens' theorem is generalized along these lines when the flow map $\phi_t$ is defined on an open subset $U$ of Euclidean space.

In a practical setting, only the time-series projection $\{h(\phi_t(x))\}_{t\geq 0}$ is available, and neither $d$ nor $d_A$ is known a priori. Additionally, the time series may be corrupted by noise. Therefore, optimally selecting the embedding dimension $m\in \mathbb{N}$ and the time delay $\tau > 0$ from this limited information is crucial  for ensuring the usefulness of the delay map in applications. Various data-driven techniques have been explored to determine these parameters numerically from the time series $\{h(\phi_t(x))\}_{t\geq 0}$, including False Nearest Neighbors~\citep{rhodes1997false}, Cao's method~\citep{cao1997practical}, mutual information~\citep{fraser1986independent,martin2024robust}, and approaches based on persistent homology~\citep{tan2023selecting}. In this work, we will assume that the time delay $\tau$ and the embedding dimension $m$ have already been chosen using these techniques.

\subsection{Optimal Transport and the Wasserstein Geometry}

A central goal of this work is to lift the statement of Takens' embedding theorem (Theorem~\ref{thm:classic_takens}) to the space of probability measures over the underlying manifold $M$. Using tools from optimal transport theory, we will rigorously define the equivalent notions, in a measure-theoretic setting, to those presented in Theorem~\ref{thm:classic_takens}. 

The space of probability measures on $M$ is given by $\mathcal{P}(M) = \{\mu \in \mathcal{B}(M) : \mu(M) = 1\}$, where by $\mathcal{B}(M)$ we denote the space of all Borel measures over $M$. Any map $h:M \rightarrow N$ can be lifted to a map from $\mathcal{P}(M) $ to $\mathcal{P}(N)$ by the pushforward operator defined below.
    \begin{definition}[The pushforward operator {\cite[Sec.~5.2]{gradient_flows}}]
        The pushforward operator lifts maps from $M$ to $N$ to maps between the equivalent spaces of probability measures and is defined by
        \begin{equation}
            h \# \mu (B) = \mu(h^{-1}(B)), \ \forall h \in \mathcal{F}(M , N), \,\, \forall B \in \mathcal{B}(N)\,.
        \end{equation}
        Equivalently, 
         $
             \int_M r(h(x))\, \rd \mu(x) = \int_N r(x) \, \rd(h\#\mu) (x)$,
         for every bounded, Borel measurable function $r:N \rightarrow \mathbb{R}$. %
    \end{definition}
    
  \subsubsection{The Continuity Equation}
    For the rest of this paper, we will restrict our study to the space of probability measures that have finite second-order moments, $\mathcal{P}_2(M) = \{\mu \in \mathcal{P}(M) \ : \ \int |x|^2 \,\rd\mu < \infty\}$. In this space, we can use the differential structure of the quadratic Wasserstein metric.  Let us consider a curve through $\mathcal{P}_2(M)$, which will be the Eulerian equivalent to the Lagrangian flow $\phi_t$ on $M$ in Theorem~\ref{thm:classic_takens}. 
    
    A first requirement for $\phi_t$ to be a flow is continuity. In the context of flows on $\mathcal{P}_2(M)$, the equivalent notion is absolute continuity. To achieve this, we require $M$ to be a metric space and assume $\mathfrak{D}:\mathcal{P}_2(M) \times \mathcal{P}_2(M) \rightarrow \mathbb{R}_{\geq 0}$ is a distance between probability measures. In particular, if $M$ is a smooth compact manifold, there always exists a metric for $M$ and we can view $M$ as a metric space. 

    \begin{definition}[Absolute continuity of curves and maps~{\cite[Definition~1.1.1]{gradient_flows}}]\ \vspace{-0.4cm}
    \begin{enumerate}
        \item We say a curve $\rho_t:[0, 1]\rightarrow \mathcal{P}_2(M)$ is absolutely continuous if $\mathfrak{D}(\rho_t, \rho_s) \leq \int_t^s f(x)\,\rd x$, where $f$ is an $L^1$ function from $[0, 1]$ to $\mathbb{R}$. %
        \item A map $F: \mathcal{P}_2(M) \rightarrow \mathcal{P}_2(N)$ is absolutely continuous if for every absolutely continuous curve $\rho_t \in \mathcal{P}_2(M)$, $F(\rho_t)$ is absolutely continuous in $\mathcal{P}_2(N)$, up to redefining $t \mapsto \rho_t$ on a set of zero Lebesgue measure on $[0,1]$. 
    \end{enumerate}
    \end{definition}
    We will use the shorthand notation ``AC'' for ``absolute continuous'' hereafter. A useful property of AC curves is that they are metric differentiable. Later, we will employ the metric derivative to bound the norm of tangent vectors (see Proposition \ref{prop:vectors}). 
    \begin{definition}[Metric derivative~{\cite[Theorem~1.1.2]{gradient_flows}}]\label{def:metric_derivative}
         For any AC curve $\rho_t \in \mathcal{P}_2(M)$, the limit
         \begin{equation}\label{eq:metric_derivative}
             |\rho_t'| = \lim_{s \to t}\frac{\mathfrak{D}(\rho_t, \rho_s)}{|t - s|}
         \end{equation}
    exists a.e.~in $t$, and the map $t\mapsto|\rho_t'|$ is $L^1$. We call $|\rho_t'|$ the metric derivative of $\rho_t$.
    \end{definition}

   In Theorem~\ref{thm:classic_takens}, the flow $\phi_t$ is generated by a smooth vector field $v:M \rightarrow TM$, such that at every point, $\frac{d}{dt}\phi_t(x) = v(x)$. In contrast, in the Wasserstein space,  the curve $\rho_t$ is generated by a square-integrable vector field
       \begin{align*}
              v_t &\in L^2(TM, \rho_t) = \Big\{v_t: M \rightarrow TM,\, t \in [0, T],\, \\ &\|v_t\|_{L^2(\rho_t)}:= \sqrt{\int_M g_x(v_t(x), v_t(x))\rd \rho_t(x) } < \infty\Big\},
       \end{align*}
    where $g_x$ is the Riemannian metric on $M$. Note that $v_t(x) \in T_xM$ and $g_x: T_xM \times T_xM\rightarrow \mathbb{R}$. 
   This vector field $v_t$ has to additionally satisfy the continuity equation
\begin{equation}\label{eq:continuity}
   \frac{\partial \rho_t}{\partial t} + \nabla\cdot (v_t\rho_t) = 0\,,
   \end{equation}
in the distributional sense, i.e.,
\begin{equation}\label{eq:continuity_dist}
       \int_0^T  \int_M \left(\frac{\partial \varphi(x, t)}{\partial t} + v_t(x)\cdot \nabla  \varphi(x, t)\right) \rd \rho_t(x) \rd t = 0,
   \end{equation}
   for all $\varphi \in \mathcal{D}(M \times [0, T])\,$
  where $\mathcal{D}(M \times [0, T])$ denotes the space of test functions on $M \times [0, T]$ (see Remark \ref{rmk:test_functions}).  This continuity equation comes from the conservation of mass. In the rest of the paper, when we say that the continuity equation is satisfied, we mean~\eqref{eq:continuity_dist} holds. Theorem~\ref{thm:ODE to PDE} below shows that any AC curve is a solution to~\eqref{eq:continuity_dist} for an appropriate choice of $v_t$~\cite[Chapter 8]{gradient_flows}. %

   \begin{theorem}\label{thm:ODE to PDE}
        Let {$v_t\in L^2(TM, \mu_t)$} be a Borel vector field such that: 
    \begin{enumerate}
        \item[(i)] $\int_0^T \int_M |v_t(x)| d\mu_t(x) \,\rd t < \infty$\,,
        \item[(ii)] $\int_0^T \sup\limits_{{B\subset M}} \left( |v_t| + \text{Lip}(v_t, B)\right)\, \rd t < \infty$\,,
    \end{enumerate}
    where $\text{Lip}(v_t, B)$ denotes the Lipschitz constant of $v_t$ on the set {$B \subset M$}. Let $\phi:M \times [0, T]\rightarrow M$  be the solution to the ODE 
\begin{equation}\label{eq:ODE}
        \phi(x, 0) = x, \qquad \frac{d}{dt}\phi(x, t) = v_t\left(\phi(x, t)\right).    \end{equation}
We will denote $\phi_t(x):= \phi(x, t) $. Let $I(x)\subset \mathbb{R}$ be the time interval on which solutions to~\eqref{eq:ODE} starting with initial position $x$ exist. Suppose that $[0, T] \subseteq I(x)$. Then,  
    \begin{enumerate}
        \item[a.] (Existence) $\rho_t = \phi_t \# \rho_0$ solves the continuity equation~\eqref{eq:continuity} with initial condition $\rho_0$. 
    \item[b.] (Uniqueness) Given any $\rho \in \mathcal{P}_2(M)$, $\rho_t = \phi_t \#\rho$  is the unique solution to the continuity equation with the initial condition $\rho$. 
    \end{enumerate}
   \end{theorem}
\begin{remark}
    This theorem establishes the connection between the deterministic dynamical system on $M$ and its lifted version on $\mathcal{P}_2(M)$. Thus, it enables us to lift the differential structure to $\mathcal{P}_2(M)$ and show that tangent vectors indeed exist. For more details see Section \ref{sec:tangent_space}.
\end{remark}
    \begin{remark}\label{rmk:test_functions}
        Without loss of generality, we will work with test functions that are compactly supported and continuously differentiable on $M \times [0, T]$, i.e., $\mathcal{D}(M\times [0, T]) = \mathcal{C}^1_c(M \times [0, T])$. This gives us sufficient regularity to define all the distributional {derivatives of $\rho$ and $v_t$ (see~\cite[Remark 8.1.1]{gradient_flows}). No generality is lost because, } if \eqref{eq:continuity_dist} holds for any $\varphi \in \mathcal{D}(M \times [0, T])$, it also holds for $\varphi \in \mathcal{C}^\infty_c(M \times [0, T])$. Similarly, for time-independent test functions, we can choose $\mathcal{D}(M) = \mathcal{C}^1_c(M)$. The same extends to test functions for $N$ and $N \times [0, T]$. 
    \end{remark}
    \subsubsection{The Tangent Space of $\mathcal{P}_2(M)$}\label{sec:tangent_space}
    Hereafter, we assume $(M, g)$ is a Riemannian manifold and $g$ is the Riemannian metric. The tangent space to a manifold at a specific point is comprised of all possible infinitesimal directions of motion starting from that point. The continuity equation~\eqref{eq:continuity} tells us how a measure evolves in time along the direction specified by a vector field $v_t$, and Theorem \ref{thm:ODE to PDE} verifies the existence of such directions.
    Hence, intuitively, we may define $T_\rho\mathcal{P}_2(M)=L^2(TM, \rho)$ as the tangent space to $\mathcal{P}_2(M)$ at $\rho$. However, not every element in $L^2(TM, \mu)$ generates a different flow in $\mathcal{P}_2(M)$. 
\begin{remark}[Non-uniqueness]\label{rmk:non uniqueness}
    Fix an AC curve $\rho_t$ and consider $v_t, w_t \in L^2(TM, \rho_t)$ where $\nabla \cdot (w_t \rho_t) = 0$ and $v_t$ satisfies  $\partial_t \rho_t + \nabla \cdot (v_t \rho_t) = 0$ (the existence of $v_t$ is guaranteed by~Theorem \ref{thm:ODE to PDE}). Then one can check that $\rho_t$ also satisfies $\partial_t \rho_t + \nabla \cdot ((v_t + w_t) \rho_t) = 0$. Hence, any two vector fields that differ by a $\rho_t$-weighted divergence-free component produce the same solution to the continuity equation. We say a vector field $v$ is  $\rho$-weighted divergence-free if $\nabla \cdot (v \rho) = 0$ holds in the distributional sense, i.e., $\forall \varphi \in \mathcal{D}(M): \int_M g_x \left( \nabla\varphi(x),  v(x) \right)\, \rd \rho(x) = 0$. A time-dependent vector field $v_t$ is divergence free if the above holds for a.e. $t$. 
\end{remark}
To choose a unique element that specifies the infinitesimal direction in $\mathcal{P}_2(M)$, we need to project out the $\rho$-weighted divergence-free component. This leads to the following definition of the Wasserstein tangent space.
\begin{definition}[Wasserstein tangent space]\label{def:tangent_space} The three definitions below are equivalent~\citep{vilani_old_and_new}.
\begin{enumerate}
    \item  The tangent space  $T_\rho
    \mathcal{P}_2(M) = L^2(TM, \rho)/\sim$, where $v\sim w \iff \forall \varphi \in \mathcal{D}(M)$,
    \begin{equation}\label{eq:equality2}
         \int_M g_x \left( \nabla\varphi(x),  v(x)-w(x) \right)\, \rd \rho(x) = 0.
    \end{equation}
    \item $ T_\rho\mathcal{P}_2(M) = \{v \in L^2(TM, \rho): \|v + w\|_{L^2(\rho)} \geq \|v\|_{L^2(\rho)}$, $\forall w\in L^2(TM, \rho) \text{ such that } \nabla\cdot (w\rho) = 0\}$.
    \item $T_\rho\mathcal{P}_2(M)$ is the closure of the space of all gradients of test functions on $M$ in $L^2_{\rho}$, i.e.,
    $$T_\rho \mathcal{P}_2(M) = \overline{\{\nabla \varphi, \ \varphi \in C_c^\infty(M)\}}^{L^2_{\rho}}.$$
\end{enumerate}
\end{definition}
\begin{remark}[Equality in $T_\rho\mathcal{P}_2(M)$]\label{rmk:equality}
    Based on Definition~\ref{def:tangent_space},  $v, w\in T_\rho \mathcal{P}_2(M)$ are equal if~\eqref{eq:equality2} holds. Equivalently, $v \neq w$  in $T_\rho \mathcal{P}_2(M)$ if
    $\exists \varphi \in \mathcal{D}(M)$ such that $$
    \int_M g_x(\nabla\varphi(x), v(x) - w(x)) \rd\rho(x) \neq 0\,.
    $$
\end{remark}

\subsubsection{The Projection Operator}
For any $w \in L^2(TM, \rho)$, one can obtain an element of $T_\rho\mathcal{P}_2(M)$ by projection: 
\begin{equation}\label{eq:projection}
w  = v + v^{\perp},\,\,\,  P^\rho(w):=v \in T_\rho \mathcal{P}_2(M),\,\,\, v^\perp \in T^\perp_\rho \mathcal{P}_2(M),\,
\end{equation}
where $T_\rho^\perp\mathcal{P}_2(M)$ is the orthogonal complement of $T_\rho \mathcal{P}_2(M)$ with respect to the $\rho$-weighted $L^2$ inner product, and $P^\rho: L^2(TM, \rho) \rightarrow T_\rho \mathcal{P}_2(M)$ is the projection operator. 
We will use this projection operator to obtain the unique vector field defining the derivative of a map between probability spaces. The kernel of $P^\rho$ comprises all $\rho$-weighted divergence-free vector fields $w$ such that $\nabla\cdot (\rho w) = 0$ in the distributional sense. One important property of $P^\rho$ is given in Proposition~\ref{prop:integral_not_changed}, whose complete proof is in Appendix~\ref{sec:proofs}.

\begin{prop}\label{prop:integral_not_changed}
    Let $(M, g)$ be a Riemannian manifold. Then for all $v\in L^2(TM,\rho)$ and any $\varphi \in\mathcal{D}(M)$,
    $$ \int _M g_x(\nabla\varphi(x) , P^\rho v(x))\rd \rho(x) = \int_M g_x(\nabla\varphi(x), v)\rd \rho(x).$$
\end{prop}

\subsubsection{Metric Change-of-Variables Formula}
We review the following result from Riemannian geometry, which will be used to show the injectivity of the metric derivative operator (see the proof of Theorem~\ref{thm:main_embedding}). A complete proof appears in Appendix~\ref{sec:proofs}. 
\begin{prop}\label{prop:change_of_vars}
    Let $(M, g)$ and $(N, q)$ be Riemannian manifolds, $f:M\rightarrow N$ be a differentiable map %
    and $\varphi \in \mathcal{D} (N)$. Then for all $X\in TM$,
    \begin{equation}\label{eq:change_of_vars}
     g_x(\nabla(\varphi\circ f)(x), X(x)) = q_{f(x)}\Big( (\nabla\varphi)(f(x)), df_x X(x) \Big),
    \end{equation}
where $df_x:T_xM \to T_{f(x)}N$ is the derivative operator of $f$ at position $x$.
\end{prop}

\section{Measure-Theoretic Time-Delay Embedding}\label{sec:measure_embedding}
In this section, we establish our main theoretical result on the measure-theoretic time-delay embedding.
In the classic Takens' embedding (Theorem~\ref{thm:classic_takens}), there are three key components: (1) the flow $\phi_t$ generated by the vector field $v$, (2) the observable $h$, and (3)~the notion of an embedding. To extend Takens' embedding to probability measures, we will find the equivalent objects to each of these three components in the space of probability distributions. 

\subsection{Differentiable Curves in $\mathcal{P}_2(M)$}
On the Lagrangian level, given an initial condition $x_0$, the flow $\phi_t(x_0)$ generates a differentiable curve $x(t) = \phi_t(x_0)$ in $M$.  Similarly, on the Eulerian level, we want to have  a differentiable curve $\rho_t$ and a vector field $v_t \in T_{\rho_t}\mathcal{P}_2(M)$ such that they satisfy the continuity equation~\eqref{eq:continuity}. To begin with, we define the notion of a vector field along a curve in $\mathcal{P}_2(M)$.

\begin{definition}[Vector fields along the curves]\label{def:vector_field}
Consider a curve $\rho_t \in \mathcal{P}_2(M)$. We say that $v_t:[0, 1] \rightarrow T_{\rho_t}\mathcal{P}_2(M)$ is a vector field along the curve $\rho_t$ if the tuple $(\rho_t, v_t)$ satisfies the continuity equation~\eqref{eq:continuity}.  If such a vector field exists, we denote it by $v_t:= \frac{d}{dt}\rho_t$.

\end{definition}

This allows us to define differentiable curves in $\mathcal{P}_2(M)$.

\begin{definition}[Differentiable curve]
    A curve $\rho_t$ in $\mathcal{P}_2(M)$ is differentiable if there exists a vector field $v_t \in T_{\rho_t}\mathcal{P}_2(M)$ along $\rho_t$ such that $\int_0^1\|v_t\|_{L^2(\rho_t)} \rd t< \infty$.  %
\end{definition}

Under this definition, the differentiable curve $\rho_t$ and the vector field $v_t$ become the analogous notions to the flow $\phi_t$ and vector field $v$ from the classical setting. We conclude this subsection by establishing its connection to AC curves. 

\begin{lemma}\label{lem:diff_curve}
    Any AC curve in $\mathcal{P}_2(M)$ is differentiable. 
\end{lemma}
The proof appears in Appendix~\ref{sec:proofs} and relies on the following Proposition:
\begin{prop}[Existence of vectors along AC curves~\citep{user_guide}]\label{prop:vectors}
    If the curve $\rho_t$ is AC, then there exists Borel vector field $v_t$ with $\|v_t\|_{L^2(\rho_t)} \leq |\rho_t'| < \infty$ a.e.~in $t$ such that~\eqref{eq:continuity} holds. %
\end{prop}

\subsection{Differentiable Maps on $\mathcal{P}_2(M)$}
The next step is to define differentiability for a map $F:\mathcal{P}_2(M) \rightarrow \mathcal{P}_2(M)$, a necessary property for $F$ to be an embedding (see Definition~\ref{def:embeddin}). In the classical sense, a map $f$ between two vector spaces is differentiable if there exists a linear operator $df$ such that 
$$
\lim_{h \to 0} \frac{|f(u + h) - f(u) - df(u)|}{|h|} = 0, \ \forall  u\in M\,.
$$ 
Moreover, a map between two manifolds is differentiable if it is locally differentiable in any chart. 
This definition cannot be directly translated to $\mathcal{P}_2(M)$ since it involves a pointwise evaluation of the differential map in any given chart, whereas the tangent vectors in $\mathcal{P}_2(M)$ are only defined almost everywhere. Therefore, we use an equivalent definition of differentiability{~\cite[Chapter 3]{leemanifolds}.}
        \begin{definition}[Differentiable maps]\label{def:differentiable}
            An absolutely continuous map $F:\mathcal{P}_2(M) \rightarrow\mathcal{P}_2(N)$ is differentiable if for any $\rho \in \mathcal{P}_2(M)$ there exists a bounded linear map $dF_\rho : T_\rho\mathcal{P}_2(M) \rightarrow T_{F(\rho)}\mathcal{P}_2(N)$ such that for any differentiable curve $\rho_t$ through $\rho$, with $\frac{d}{dt}\rho_t = v_t$, the curve $F(\rho_t)$ is differentiable. Moreover, the derivative operator of $F$ is $dF_{\rho_t}: T_{\rho_t}\mathcal{P}_2(M) \to T_{F(\rho_t)}\mathcal{P}_2(N)$,  $dF_{\rho_t} (v_t) :=\frac{d}{dt}F(\rho_t)  $.
        \end{definition}
        In other words, Definition~\ref{def:differentiable} requires that the map $F$ takes differentiable curves to differentiable curves, and tangent vectors to the corresponding tangent vectors. 
   
   Next, we consider the particular situation where the map $F:\mathcal{P}_2(M) \rightarrow \mathcal{P}_2(N)$ is the pushforward of some $f:M \rightarrow N$. This is exactly the case for the measure-theoretic delay-embedding map $\Psi_{h, \phi_{\tau}} := \Phi_{h, \phi_{\tau}} \#$. Since the classic delay-embedding map $\Phi_{h, \phi_{\tau}}$ is invertible, we will specifically analyze invertible $f$. A generalization of Theorem~\ref{thm:push_forward_differentiable} below can be found in~\citep{diff_maps}. For completeness, we provide a full proof of Theorem \ref{thm:push_forward_differentiable} in Appendix~\ref{sec:proofs}.

    \begin{theorem}[The pushforward map is differentiable]\label{thm:push_forward_differentiable}
        Let $F = f\#$ as described above and assume $f:(M, g) \rightarrow (N, q)$ is continuously differentiable, invertible  and proper such that $\sup_{x\in M} \|df_x\|< \infty$. Then $F$ is differentiable (in the sense of Definition \ref{def:differentiable}) and $dF_\rho = P^{F(\rho)}\widetilde{dF}_{\rho}$, where
        \begin{equation}\label{eq:tilde_psi}
          \widetilde{dF}_{\rho} (v)(y) := df_{f^{-1}(y)}(v(f^{-1}(y))), \ \forall y \in N, \ \forall v \in T_\rho \mathcal{P}_2(M)\,.
        \end{equation}
    \end{theorem}

    \subsection{The Embedding in $\mathcal{P}_2(M)$}
    Building on top of Definition~\ref{def:differentiable}, we turn to the notion of an embedding in the space of probability distributions. Intuitively, an embedding is a diffeomorphism which preserves the differential structure (see Definition~\ref{def:embeddin}). In our situation, this  structure is given by the geometry of $T\mathcal{P}_2(M)$ described in Section~\ref{sec:tangent_space}.
    \begin{definition}[Embedding in the probability space]\label{def:embedding}
    A map $F:\mathcal{P}_2(M) \rightarrow \mathcal{P}_2(N)$ is an embedding if the following conditions are satisfied:
    \begin{enumerate}
        \item[(i)] $F$ is a bijection onto its image, i.e., $\forall \rho, \eta \in \mathcal{P}_2(M)$ such that $\rho \neq \eta$ as probability distributions, $F(\rho) \neq F(\eta)$;
        \item[(ii)] $F$ is differentiable in the sense of Definition~\ref{def:differentiable};
        \item[(iii)] The derivative operator $dF$ is injective, i.e., for any $v, w \in T_\rho \mathcal{P}_2(M)$ such that $v \neq w$ (in the sense of~\eqref{eq:equality2}), $dF(v) \neq dF(w)$ as vectors in $T_{F(\rho)}\mathcal{P}_2(N)$.
    \end{enumerate}
\end{definition}
\subsection{Statement of the Main Theorem}
Having defined all the prerequisites, we are now ready to state the measure-theoretic version of time-delay embedding theorem.

\begin{theorem}\label{thm:main_embedding}
Let $f:M \to N$ be an embedding between two differentiable manifolds $M$ and $N$. Then the map $F := f\# :\mathcal{P}_2(M)\rightarrow \mathcal{P}_2(N)$ is an embedding between the spaces of probability distributions on $M$ and $N$, respectively (in the sense of Definition~\ref{def:embedding}).
\end{theorem}

A direct corollary of this theorem gives us the measure-theoretic time-delay embedding.
\begin{cor}\label{cor:1}
     Let $\phi_t:M \rightarrow M$ be a dynamical system on a compact $d$-dimensional manifold $M$ such that its vector field satisfies the conditions of Theorem~\ref{thm:classic_takens}. For an observable $h\in \mathcal{C}^2(M,\mathbb{R})$, define the delay embedding map $\Phi_{h, \phi_{\tau}}:M \rightarrow \mathbb{R}^m$ as in~\eqref{eq:delay} and let  $\Psi_{h, \phi_{\tau}}: \mathcal{P}_2(M) \rightarrow \mathcal{P}_2(\mathbb{R}^m)$ be its pushforward, i.e., $\Psi_{h, \phi_{\tau}} \rho = \Phi_{h, \phi_{\tau}}\# \rho$. Then, if $m \geq 2d + 1$, it is a generic property that $\Psi_{h, \phi_{\tau}}$ is an embedding of $\mathcal{P}_2(M)$ into $\mathcal{P}_2(\mathbb{R}^m)$ (in the sense of Definition~\ref{def:embedding}).
\end{cor}
\begin{proof}[Proof of Corollary~\ref{cor:1}]
    Since Theorem \ref{thm:classic_takens} holds, $\Phi_{h, \phi_{\tau}}$ is generically an embedding. Hence, Theorem \ref{thm:main_embedding} can be applied to deduce that $\Psi_{h, \phi_{\tau}} = \Phi_{h, \phi_{\tau}}\#$ is generically an embedding between $\mathcal{P}_2(M)$ and $\mathcal{P}_2(\mathbb{R}^m)$.
\end{proof}

\begin{proof}[Proof of Theorem~\ref{thm:main_embedding}]
We will show that the three conditions in Definition~\ref{def:embedding} are satisfied. We start by showing that $F$ is injective. Assume $\rho_0,\rho_1 \in \mathcal{P}_2(M)$ such that $F(\rho_0) = F(\rho_1)$. By the definition of $F$, we have
\begin{equation}\label{eq:injectivity}
    f \# \rho_0 = f\# \rho_1.
\end{equation}
Since $f$ is a bijection, there exists the inverse map $f^{-1}:f(M) \subset N \rightarrow M$  such that  $f^{-1} \circ f = \textit{Id}_M$. Consequently, $F$ is injective as
$$f^{-1}\#f \# \rho_0 = f^{-1}\#f \# \rho_1 \iff \rho_0 = \rho_1.$$

Differentiability of $F$ follows from Theorem~\ref{thm:push_forward_differentiable} because $F$ is the pushforward of $f$, with the latter being an invertible and proper map. Additionally, \eqref{eq:tilde_psi} gives an explicit formula for the derivative operator $dF : T\mathcal{P}_2(M) \rightarrow  T\mathcal{P}_2(N)$, 
\[
    dF_{\rho} = P^{F(\rho)}\widetilde{dF}_{\rho} , \text{ where } \widetilde{dF}_{\rho} (y) = df_{f^{-1}(y)}(v(f^{-1}(y))).
\]

Lastly, we show that the derivative operator is injective, i.e., if $v \neq w$ in $T_{\rho}\mathcal{P}_2(M)$, then $dF_{\rho}(v) \neq  dF_\rho(w)$ in $T_{F({\rho})} \mathcal{P}_2(N)$.
By Remark~\ref{rmk:equality}, $v \neq w$ implies $\exists \varphi \in \mathcal{D}(M)$ such that
\begin{equation}\label{eq:what_we_know}
    \int_M g_x(\nabla\varphi, v - w) \rd \rho_t(x) \neq 0.
\end{equation}
The goal is to find a test function $\varsigma \in \mathcal{D}(N)$ such that 
\[
    \int_{\mathbb{R}^d}q_y\Big( \nabla \varsigma (y), dF_{\rho}(v) (y) - dF_{\rho}(w) (y)\Big) \rd \nu(y)\neq 0,
\] 
where $\nu = F({\rho}) = f \#\rho$, and $q_y$ denotes the Riemannian inner product in $N$ at the point $y\in N$. Further derivation shows that
\begin{align*}
     &\int_{N}q_y\Big(  \nabla \varsigma (y), dF_{\rho}(v) (y) - dF_{\rho}(w) (y)\Big) \rd \nu(y)  &\text{ By the linearity of }dF_\rho\\
    = & \int_{f(M)}q_y\Big( \nabla \varsigma (y), dF_{\rho}(v - w) (y)\Big)  \rd \nu(y) &\text{ Since } dF_\rho - \widetilde{dF}_\rho \text{ is divergence free}\\
    =& \int_{f(M)}q_y\Big( \nabla \varsigma (y),\widetilde{dF}_{\rho} (v - w) (y)\Big)  \rd \nu(y)&\text{ Using Theorem \ref{thm:push_forward_differentiable}} \\
    =& \int_{f(M)}q_y\Big( \nabla \varsigma (y), df_{f^{-1}(y)} (v - w) (f^{-1}(y))\Big) \rd(f \#\rho)(y) \\
    =& \int_M q_{f(x)}\Big(\nabla\varsigma (f(x)), df_x (v - w)(x)\Big) \rd\rho(x) & \text{Using Proposition~\ref{prop:change_of_vars}}\\
    =& \int_M g_x(\nabla (\varsigma \circ f)(x), (v - w)(x)) \rd\rho(x). 
    \end{align*}
Since $f$ is an embedding,  it is differentiable and invertible. Plugging $\varsigma = \varphi \circ f^{-1} : N \rightarrow \mathbb{R}$ back into the last equation above, we get \eqref{eq:what_we_know}. 

The only claim left to show is that the $\varsigma$ defined above lies in $\mathcal{D}(N)$. To show $\varsigma$ is compactly supported, note that
\begin{align*}
\{x : \varsigma(x) \neq 0\}
&= \{x : \varphi(f^{-1}(x)) \neq 0\} 
\\
& = \{x : f^{-1}(x) \in \{y : \varphi(y) \neq 0\}\} \\
& = f\bigl(\{y : \varphi(y) \neq 0\}\bigr).
\end{align*}
Taking closures and using the fact that a homeomorphism preserves closures, we get $\text{supp}(\varsigma) = f(\text{supp}(\varphi))$. Also, since $f$ is a homeomorphism, it maps compact sets to compact sets, so $\text{supp}(\varsigma)$ is compact. Moreover, $\varsigma \in \mathcal{C}^1 (N)$ because $\varphi\in \mathcal{C}^1$ and $f^{-1} \in \mathcal{C}^1(M, N)$. We conclude our proof with $\varsigma \in \mathcal{D}(N)$.
\end{proof}

\section{Numerical Experiments}\label{sec:experiments}
In this section, we introduce a measure-theoretic computational framework for learning the full-state reconstruction map as a pushforward between probability spaces\footnote{Our code is available at \href{https://github.com/jrbotvinick/Measure-Theoretic-Time-Delay-Embedding}{https://github.com/jrbotvinick/Measure-Theoretic-Time-Delay-Embedding}. }.   In Section~\ref{subsec:motivation}, we leverage Theorem \ref{thm:main_embedding} to introduce and motivate our proposed methodology.  In Section \ref{subsec:noisydata}, we demonstrate the robustness of our learned reconstructions to extrinsic noise in the training data for synthetic test systems. In Section~\ref{subsec:realdata}, we combine our framework with POD-based model reduction to  reconstruct the NOAA Sea Surface Temperature (SST) dataset from partial measurement data. Finally, in Section~\ref{subsec:ERA5} we reconstruct the ERA5 wind-speed dataset from partially observed vector-valued data. Throughout, all experiments are conducted using an Intel i7-1165G7 CPU. 

\subsection{From Theory to Applications}\label{subsec:motivation}
\subsubsection{Motivation}
While the classical Takens' embedding theorem guarantees the existence of a reconstruction map from delay coordinates to the full state (see Theorem \ref{thm:classic_takens}), it provides no analytic form for the function.
In applications, the resulting coordinate transformation is often learned from data. However, the accuracy of these learning methods can be significantly compromised  when the available data is noisy and sparse. To address this issue, we develop a measure-theoretic approach to learning the reconstruction map, inspired by Corollary~\ref{cor:1}.

We begin  by considering samples $\{x_i\}$ along a trajectory of a smooth flow $\phi_t:M\to M$, where $M\subseteq \mathbb{R}^n$  is a smooth compact $d$-dimensional manifold. In applications, samples $\{x_i\}$ of the full state can rarely be observed directly, and instead, one may only have access to the values $\{h(x_i)\}$ of an observable along the trajectory. For suitably chosen time-delay parameters $m\in \mathbb{N}$ and $\tau> 0$, the map $\Phi = \Phi_{h,\phi_{\tau}}$ is an embedding of $M$, and one can form the time-delayed trajectory $\{\Phi(x_i)\}$. It also holds that $\{\Phi(x_i)\}$ and $\{x_i\}$ are  related pointwisely by the smooth deterministic map $\Phi^{-1}:\Phi(M) \to \mathbb{R}^{n}$. Thus, if one can learn the reconstruction map $\Phi^{-1}$ from the paired data $\{(x_i,\Phi(x_i))\}$, then the history of the one-dimensional time-series $\{h(x_i)\}$ can be used to {reconstruct} the entire $n$-dimensional trajectory $\{x_i\}.$ 

{This approach is useful in applications where only limited training data from the full state is available  \cite{lu2017reservoir}. In particular, it may be possible to access the full state of a high-dimensional system for a short period of time, but it may be prohibitively expensive or challenging to continue measuring it directly at later times. In such a setup, one may build a surrogate mapping from delay embeddings of a partially observed state variable to the full state. Thus, expensive measurements of the full state need not be recorded, and one can instead reconstruct the full state from easier-to-access partial observations. This problem is common in the surrogate modeling of fluid flows  where one builds a full-state reconstruction map using only a few sparsely distributed sensors \cite{maulik2020probabilistic,callaham2019robust}.}

In certain applications, no data from the full state can be observed. In these cases, one can still predict the evolution of the one-dimensional time series $\{h(x_i)\}$ by learning the map 
\begin{equation}\label{eq:delay_dynamics}
    \mathcal{T}:\Phi(M)\to \Phi(M),\qquad \mathcal{T}(\Phi(x_i)) = \Phi(x_{i+1}),
\end{equation}
where here we have assumed that the samples $\{x_i\}$ are obtained uniformly in time. Then, the delay state can be predicted by iterated composition with $\mathcal{T}$, and the forecasted time series can be recovered by projecting to a single axis.

\subsubsection{The Measure-Theoretic Loss}

We now recall the standard pointwise approach for learning the reconstruction map $\Phi^{-1}$, which is used in \citep{bakarji2022discovering, raut2023arousal, young2023deep}. Given paired data  $\{(x_i,\Phi(x_i))\}_{i=1}^{N} \subseteq \mathbb{R}^{n} \times \mathbb{R}^m$, one option is to learn the reconstruction map $\Phi^{-1}$ by parameterizing $\mathcal{R}
_{\theta}: \mathbb{R}^{m}\to \mathbb{R}^{n}$ in some function space $\mathcal{F}$ and optimizing the parameters $\theta\in \Theta\subseteq \mathbb{R}^p$ using the pointwise mean-squared error (MSE) reconstruction loss
\begin{equation}\label{eq:pointwise}
    \mathcal{L}_{\text{p}}(\theta)= \frac{1}{N}\sum_{i=1}^{N} |x_i - \mathcal{R}_{\theta}(\Phi(x_i))|^2\,.
\end{equation} While \eqref{eq:pointwise} is efficient and simple to implement, it is also prone to overfitting noise in the training data, especially when the available samples are sparse and limited. 

Different from \eqref{eq:pointwise}, we propose to consider data of the form $\{(\mu_i,\Phi\# \mu_i)\}_{i=1}^{K} \subseteq \mathcal{P}_2(\mathbb{R}^{n})\times \mathcal{P}_2(\mathbb{R}^m)$ and instead use the measure-theoretic objective
\begin{equation}\label{eq:distribution}
\mathcal{L}_{\text{m}}(\theta)= \frac{1}{K}\sum_{i=1}^{K} \mathfrak{D}\left( \mu_i,\mathcal{R}_{\theta}\# (\Phi\# \mu_i)\right),
\end{equation}
where $\mathfrak{D}:\mathcal{P}_2(\mathbb{R}^{n})\times \mathcal{P}_2(\mathbb{R}^{n})\to \mathbb{R}$ is a metric or divergence on the space of probability measures. 
 Theorem~\ref{thm:main_embedding} indicates that for a suitable parameterization of $\mathcal{R}_{\theta}$, the loss \eqref{eq:distribution} can indeed be reduced to zero in a noise-free setting. Moreover, while in the pointwise loss~\eqref{eq:pointwise} we seek to recover a map between $\mathbb{R}^m$ and $\mathbb{R}^n$, in \eqref{eq:distribution} we instead search for a map between the corresponding probability spaces $\mathcal{P}_2(\mathbb{R}^m)$ and $\mathcal{P}_2(\mathbb{R}^n),$ which is parameterized by the pushforward of some function that maps $\mathbb{R}^m$ to $\mathbb{R}^n$. Theorem~\ref{thm:main_embedding} guarantees the existence of such a map between the probability spaces $\mathcal{P}_2(\mathbb{R}^m)$ and $\mathcal{P}_2(\mathbb{R}^n)$ with suitable regularity properties, i.e., it is a smooth embedding in the measure-theoretic sense discussed in Section~\ref{sec:measure_embedding}. 
 
 In applications, one commonly only has access to the pointwise data $\{(x_i,\Phi(x_i))\}_{i=1}^{N}$, and thus the measure data $\{(\mu_i,\Phi\# \mu_i)\}_{i=1}^{K}$ must be constructed based upon the pointwise data. Similar to~\citep{kirtland2023unstructured}, we  use $\texttt{k}$-means clustering to partition the time-delayed trajectory $\{\Phi(x_i)\}_{i=1}^{N}$ into Voronoi cells $\{C_i\}_{i=1}^{K}$, and we define for each $1\leq i \leq K$ the discrete measure
\begin{equation}\label{eq:measure}
    \mu_i:= \frac{1}{N_i}\sum_{j=1}^{N_i} \delta_{x^{i}_j}\in \mathcal{P}_2(\mathbb{R}^n), 
\end{equation}
where $\,\,N_i:=|\{ 1\leq k \leq N :\Phi(x_k)\in C_i\}|$ and $\{x_{j}^{i}\}_{j=1}^{N_i}$ denotes the samples in $\mathbb{R}^n$ such that $\Phi(x_{j}^{i}) \in C_i$,  $j = 1,\ldots, N_i$. If $\{ x_i\}_{i=1}^{N}$ are samples from a long trajectory whose underlying flow admits a physical invariant measure (or Sinai--Ruelle--Bowen measure~\citep{young2002srb}) $\nu$, then the measure $\mu_i$ defined in~\eqref{eq:measure} approximates $\nu|_{\Phi^{-1}(C_i)}$. This is the restriction of the measure $\nu$ to the set $\Phi^{-1}(C_i)$, i.e., a conditional distribution, provided that $\nu(\Phi^{-1}(C_i)) > 0$.

We also remark that the loss functions \eqref{eq:pointwise} and 
\eqref{eq:distribution} can be modified to account for the case when one is interested in learning the dynamics $\Phi(x_i)\mapsto \Phi(x_{i+1})$ in delay coordinates. For this related problem, the pointwise loss becomes 
\begin{equation}\label{eq:pointwise2}
    \mathcal{L}_{\textrm{p}}(\theta) = \frac{1}{N}\sum_{i=1}^N|\Phi(x_{i+1}) - \mathcal{T}_{\theta}(\Phi(x_i))|^2,
\end{equation}
where $\mathcal{T}_{\theta}$ is a neural-network parameterized flow map representing the dynamics in time-delay coordinates. Similarly, the measure-theoretic analog is 
\begin{equation}\label{eq:distribution2}
\mathcal{L}_{\text{m}}(\theta)= \frac{1}{K}\sum_{i=1}^{K} \mathfrak{D}\left( \mathcal{T}\#(\Phi\#\mu_{i}),\mathcal{T}_{\theta}\# (\Phi\# \mu_i)\right),
\end{equation}
where $\mathcal{T}$ is defined in \eqref{eq:delay_dynamics}. Notably, both \eqref{eq:pointwise2} and \eqref{eq:distribution2} can be evaluated without having direct access to $\mathcal{T}$ and without observing the full state of the dynamical system. While not explored in this work, 
one can also leverage the inherent shift-based structure of the delay-coordinate dynamics $\mathcal{T}$ to restrict the hypothesis space of flow maps in learning $\mathcal{T}_{\theta}$.

\subsubsection{Discussion}

Enforcing the measure-theoretic objective \eqref{eq:distribution} bears similarities to the approaches in~\citep{kirtland2023unstructured,araki2021grid}, where the reconstruction map $\mathcal{R}_{\theta}$ is learned by averaging the full state over each cluster in the reconstruction space and then linearly interpolating between these averages in the delay coordinate. While these methods ensure that $\mathcal{R}_{\theta}\# (\Phi\# \mu_i)$ and $\mu_i$ agree in expectation, our measure-theoretic approach is designed to match not only the expectation but also all moments of the measures (see Eqn.~\eqref{eq:distribution}).

Here, we discuss the relationship between the pointwise and measure-theoretic loss functions in more detail. If the distributional loss $\mathcal{L}_{\text{m}}$ (see Eqn.~\eqref{eq:distribution}) is reduced to zero, then in general, the pointwise loss $\mathcal{L}_{\text{p}}$ (see Eqn.~\eqref{eq:pointwise}) may still be large. As the diameter of each partition element $C_i$ decreases, this discrepancy becomes small, and in the limit when $\mu_i = \delta_{x_i}$, the loss functions $\mathcal{L}_{\text{p}}$ and $\mathcal{L}_{\text{m}}$ are equivalent for a suitable choice of $\mathfrak{D}$, e.g., the squared Wasserstein distance $\mathfrak{D} = W_2^2$. Therefore, $\mathcal{L}_{\text{m}}$ should be viewed as a relaxation of $\mathcal{L}_{\text{p}}$, where the diameter of each partition element controls the extent to which pointwise errors in the measure-based reconstruction are permitted. In practice, the partition elements' diameter should be chosen according to the number of data points and the amount of noise present; see~\cite[Fig.~6]{kirtland2023unstructured} for a similar discussion. 

It is also worth noting that there may be several minimizers of $\mathcal{L}_{\text{m}}$, depending on how the measures $\{\mu_i\}_{i=1}^K$ are constructed. In general, any minimizer of $\mathcal{L}_{\text{p}}$ is a minimizer of $\mathcal{L}_{\text{m}}$.
However, when noise is present in the training data, any minimizer of $\mathcal{L}_{\text{p}}$ will be highly oscillatory and challenging to approximate. Thus, if $\mathcal{R}_{\theta}$ is a neural network, its spectral bias creates an implicit regularization during training which will favor smoother, less oscillatory, solutions~\citep{rahaman2019spectral}.  Furthermore, it is well-established that loss functions comparing probability measures, e.g., $f$-divergence and the Wasserstein metric, are less sensitive to oscillatory noise compared to pointwise metrics like MSE~\citep{engquist2020quadratic,ernst2022wasserstein}. Hence, the minimizers of~$\mathcal{L}_{\text{m}}$ tend to exhibit better generalization properties than those of $\mathcal{L}_{\text{p}}$.

{Both the diameter of the cells $C_i$ and the number of samples $N_i$ residing in each cell contribute potential sources of error when using the measure-based loss~\eqref{eq:distribution}. As the number of samples $N_i$ increases and the diameters of the cells $C_i$ decrease, the accuracy of the measure-based reconstruction is expected to increase in an ideal noise-free setting. When the available data is noise-free, the use of empirical measures supported on finite diameter cells may unnecessarily hinder accuracy and the pointwise loss \eqref{eq:pointwise} is generally preferred. However, when the data is corrupted by noise, the pointwise loss \eqref{eq:pointwise} tends to overfit the observed data, while we find the flexibility of the measure-based matching \eqref{eq:distribution} to be more robust. We also remark that the construction~\eqref{eq:measure} of empirical measures is only one possible way of forming measure-valued data $\mu_i$ for training the loss \eqref{eq:distribution}. In this work, we  choose to consider localized measures supported on Voronoi cells, as the cell diameters can be tuned to effectively balance approximation accuracy and noise-tolerance.}

Throughout our numerical experiments, $\mathcal{R}_{\theta}$ is parameterized as a standard feed forward neural network, and the weights and biases $\theta$ are optimized using Adam \citep{kingma2014adam}. {Similarly, the delay-coordinate dynamics $\mathcal{T}_{\theta}$ are defined by the flow map of a neural network-parameterized vector field.} We choose $\mathfrak{D}$ to be the Maximum Mean Discrepancy (MMD)~\citep{gretton2012kernel} based on either the kernel {$k_{1}(x,y)= -\|x-y\|_2$} or the Gaussian kernel {$k_{2}(x,y)= \exp(-\|x-y\|^2_2/2\sigma^2)$}. We note that the MMD based upon the kernel {$k_{1}(x,y)$} is also known as the Energy Distance MMD. We use the \texttt{Geomloss} library to compute $\mathfrak{D}$, which is fully compatible with PyTorch's autograd engine \citep{feydy2019interpolating}.  There are several times in our numerical experiments where the  chosen embedding dimension $m\in \mathbb{N}$ does not satisfy the inequality $m \geq 2d+1$ introduced in Theorem \ref{thm:classic_takens}. While Theorem \ref{thm:classic_takens} provides an upper bound on the minimum required embedding dimension, it is often the case that lower-dimensional embeddings can be achieved. Several works have studied the problem of estimating such an embedding dimension numerically from partially observed time-series data. In this work, we use \texttt{teaspoon}~\citep{myers2020teaspoon} to inform our selection of the embedding parameters $\tau > 0$ and $m \in \mathbb{N}$  using both the mutual information~\citep{fraser1986independent} and Cao's method~\citep{cao1997practical}. Notably, \cite{cao1997practical} estimates the minimum embedding dimension of the Lorenz-63 system to be $m = 3$, while the upper bound in Theorem \ref{thm:classic_takens} suggests taking $m \geq 7$.
\subsection{First Examples: Noisy Chaotic Attractors}\label{subsec:noisydata}
\subsubsection{{State Reconstruction}}\label{subsubsec:recon}
We begin by studying our measure-theoretic approach to state reconstruction on the Lorenz-63 system~\citep{tucker1999lorenz}, the R\"ossler system~\citep{rossler1976equation}, and a four-dimensional Lotka--Volterra model~\citep{vano2006chaos}. For these dynamical systems, we select standard values for the systems' parameters, which are known to produce chaotic trajectories; see Appendix~\ref{sec:exps} for the system equations and precise parameter choices.

The task is to reconstruct the full state of these systems using partial observations. {In particular, we aim to learn the inverse embedding map $\Phi^{-1}$ using input-output training data from both the delay state of a partial observation and the system's full state.} For each system, we first simulate a long trajectory, form the delayed state based on a scalar observable, and split the data into training and testing components. As explained in Section~\ref{subsec:motivation}, the measures $\{\Phi\#\mu_i\}_{i=1}^K$ are given by conditioning a long trajectory in time-delay coordinates on various regions of the attractor, which in practice is implemented by a \texttt{k}-means clustering algorithm. 

We remark that both the pointwise approach~\eqref{eq:pointwise} and the measure-based approach~\eqref{eq:distribution} can achieve accurate reconstructions when the data is noise-free; see Appendix~\ref{sec:exps} for an experiment demonstrating the measure-based approach applied to clean data. However, our method proves particularly advantageous when dealing with sparse and noisy data. To demonstrate this, we compare the performance of the measure-based method with pointwise matching in Fig.~\ref{fig:noisecompare} using imperfect data. In these tests, extrinsic noise is applied to the entire state, including the time series that forms the time-delay coordinates. From these corrupted inputs and outputs, we learn the full-state reconstruction map. Although neither method is expected to achieve perfect reconstruction, the measure-based approach yields smoother results, whereas the pointwise approach tends to overfit the noise.

\begin{figure}
    \centering
   \subfloat[Lorenz-63 system reconstructions]{ \includegraphics[width = .9\textwidth]{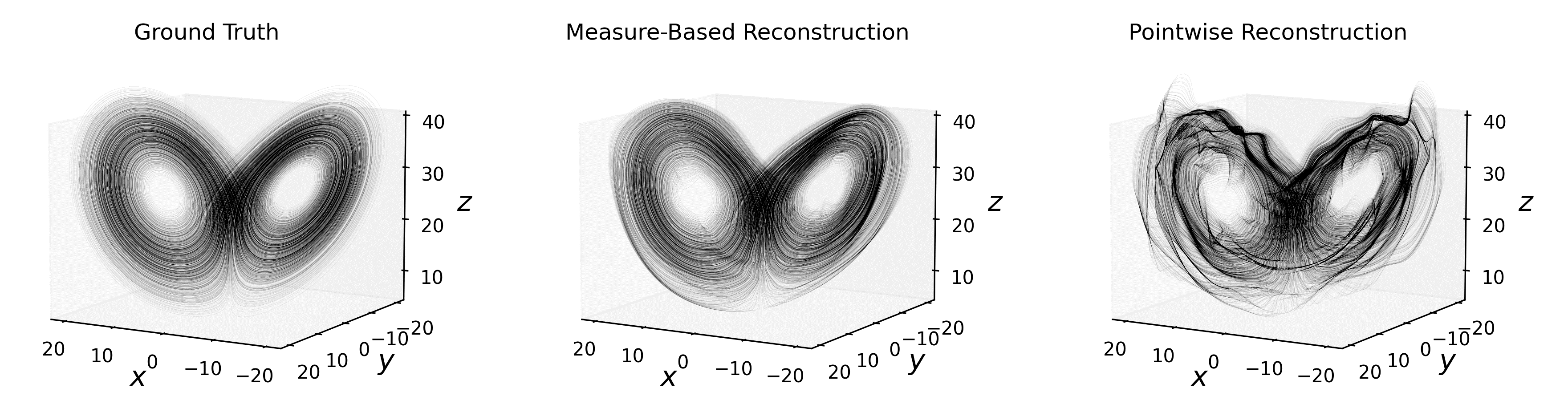}\label{fig:3a}}\\
   \vspace{.5cm}
    \subfloat[R\"ossler system reconstructions]{ \includegraphics[width = .9\textwidth]{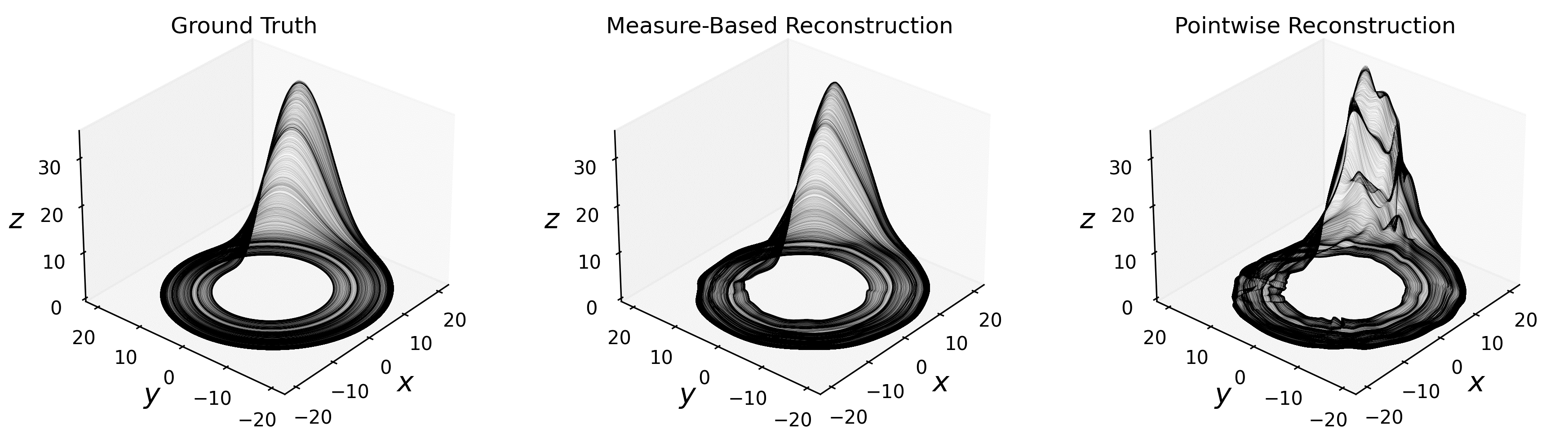}\label{fig:3b}}\\
       \vspace{.5cm}
       \subfloat[Lotka--Volterra reconstructions]{ \includegraphics[width = .9\textwidth]{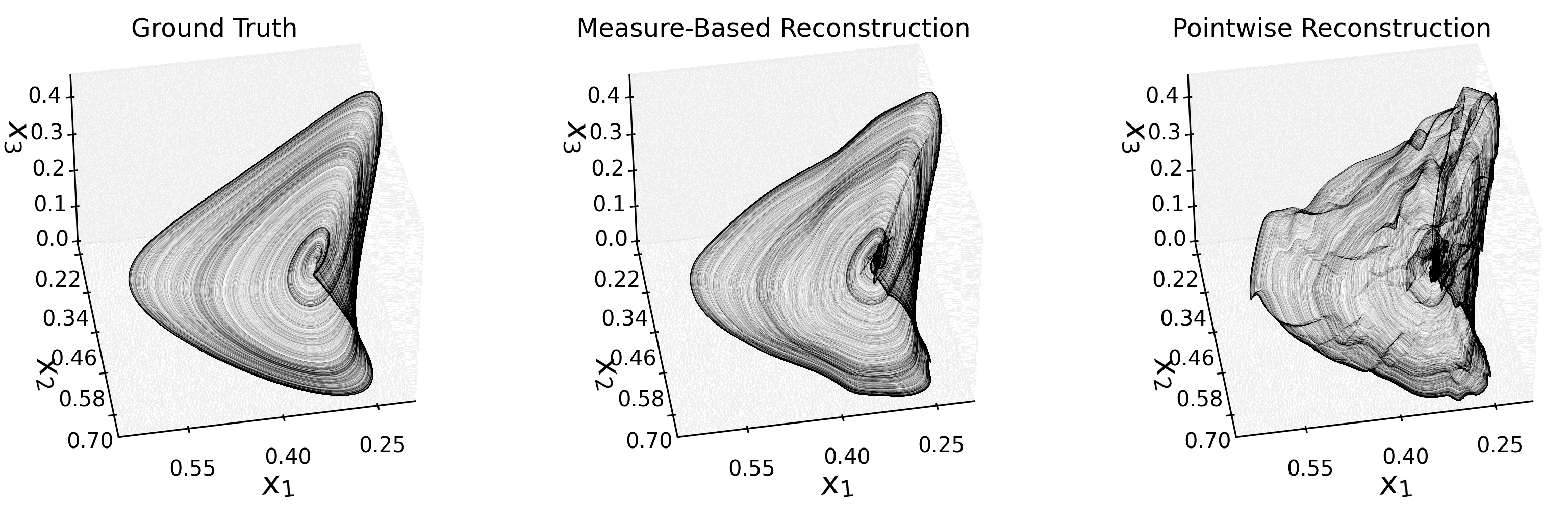}\label{fig:3c}}
    \caption{Visualizations of the learned full-state reconstruction map with sparse and noisy data for systems discussed in Section~\ref{subsec:noisydata}. The comparison includes both pointwise and measure-based approaches against the ground truth.}
    \label{fig:noisecompare}
\end{figure}

\begin{table}[h!]
    \centering
    \renewcommand{\arraystretch}{1.5}
      \begin{tabular}{|c|c|c|}
    \hline
      System & Pointwise MSE & Measure MSE \\
      \hline
     Lorenz & $7.15\times 10^{-1}$ & $2.84\times 10^{-1}$ \\ 
     R\"ossler & $3.99\times 10^{-1}$ & $8.34\times 10^{-2}$ \\
     Lotka--Volterra & $2.10\times 10^{-4}$ & $9.45\times 10^{-5}$ \\     \hline
    \end{tabular}
    \caption{Mean-squared error (MSE) for the reconstructions in Fig.~\ref{fig:noisecompare}. The measure-based reconstruction has lower error for all tests.\label{tab:1}}
\end{table}
For the experiments shown in Fig.~\ref{fig:noisecompare}, the training data consists of $2\times 10^3$ input-output pairs, which are obtained as random samples from a long trajectory. The data is corrupted with i.i.d.~extrinsic Gaussian noise samples with covariance matrices $\Sigma_{\texttt{Lorenz}} = 0.1 I$, $\Sigma_{\texttt{R\"ossler}} =0.1I $, and $\Sigma_{\texttt{Lotka-Volterra}}= 5\times 10^{-5} I.$ It is important to note that the noise in the time-delay coordinate may exhibit potential correlations, as the time series used to form these coordinates---taken as the projection of the dynamics onto the $x$-axis---is embedded after the extrinsic noise is applied. The noisy delay state is then partitioned evenly into $20$ cells via a constrained \texttt{k}-means routine~\citep{bradley2000constrained}, from which we then form noisy approximations to the measures following~\eqref{eq:measure}. Across all tests, the same four-layer neural network with hyperbolic tangent activation, $100$ nodes in each layer, and a learning rate of $10^{-3}$ is trained for $5\times 10^4$ steps. After training the networks on the noisy data, the accuracy of the learned reconstruction map is assessed by applying the network to a clean signal in the time-delay coordinate system. The MSE for the reconstructions visualized in Fig.~\ref{fig:noisecompare} is summarized in Table~\ref{tab:1}. For each experiment, the measure-based reconstruction achieves lower error.

{
For the Lorenz-63 experiment, we also estimate the Lipschitz constant of the learned inverse embedding map (parameterized by a neural network) in both cases when the data is clean and noisy. The Lipschitz constant is approximated as $\sup_{z_i}\|\nabla \mathcal{R}_{\theta}(z_i)\|_{2}$, where $\nabla \mathcal{R}_{\theta}$ is computed via automatic differentiation and $\{z_i\}_{i=1}^{10^5}$ are uniform samples from the unit cube, which contains the normalized training data for neural network training. We find that the Lipschitz constants of the pointwise approach when trained on clean and noisy data are 7.12 and 19.97, respectively. Moreover, the Lipschitz constants of the measure-theoretic approach are 7.56 and 7.09, respectively. The reported Lipschitz constants are the median values obtained from 10 neural network trainings each with different random initializations. While the Lipschitz constants of the measure-theoretic and pointwise approaches are comparable when the data is clean, the pointwise objective has a significantly larger Lipschitz constant when trained on noisy data. This observation is consistent with the oscillatory behavior of the pointwise reconstructions in Fig.~\ref{fig:3a}.}

 \subsubsection{Time-Series Prediction}\label{subsubsec:predict}
While in Section \ref{subsubsec:recon} we studied the problem of learning the reconstruction map from the delay state to the full state, we now focus on the problem of time-series prediction. In this setting, we first form the the delay-coordinates from a partially observed scalar-valued time-series. We then aim to learn the dynamics in delay coordinates using either the pointwise \eqref{eq:pointwise2} or measure-theoretic \eqref{eq:distribution2} objectives. After learning the map, we simulate the learned delay-coordinate dynamics forward in time to assess the model's accuracy. This can be turned into a scalar time-series prediction simply by projecting the modeled delay-coordinate trajectory onto a single coordinate axis.

\begin{figure}[h!]
  \centering
   \subfloat[Visualizing the 1-step prediction error in delay coordinates for both the pointwise and measure approaches when trained on noisy data.]{\includegraphics[width = .8\textwidth]{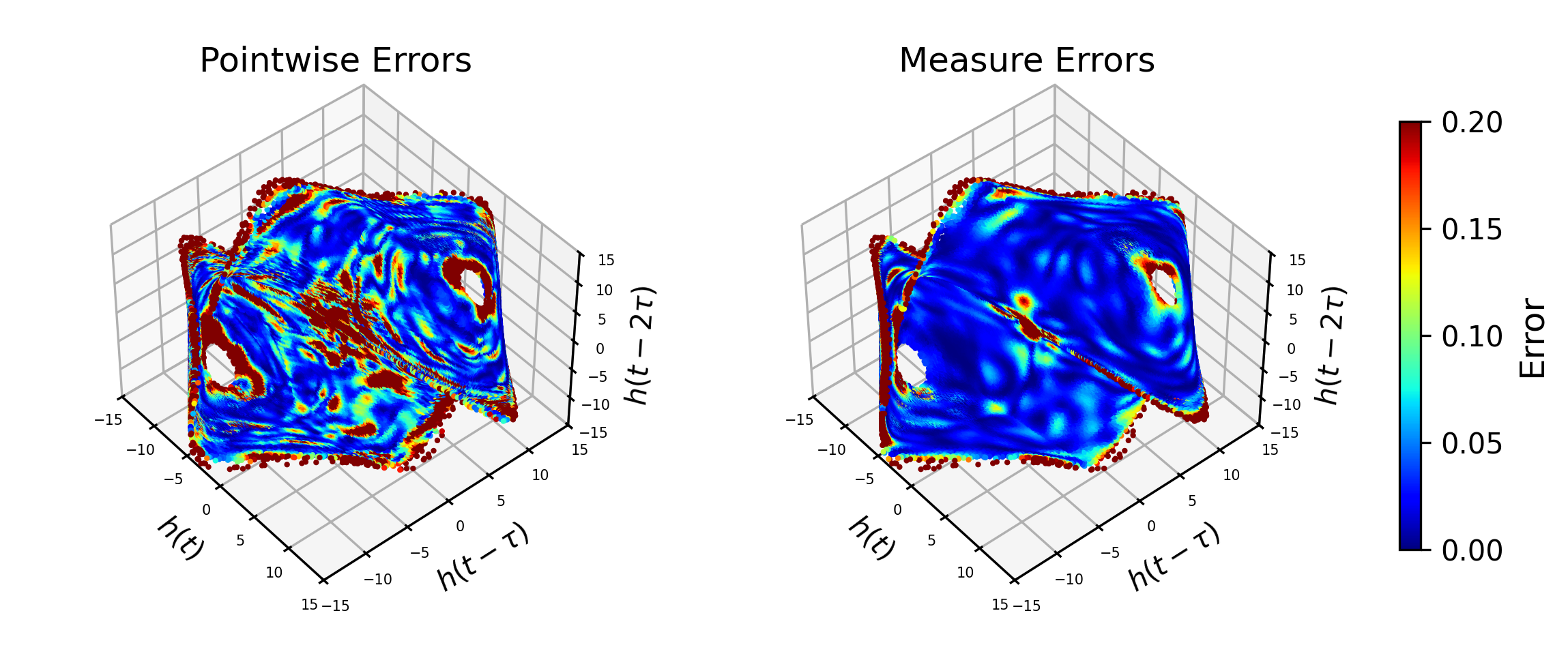}\label{fig:pred3}}\\
   \subfloat[\blue{Simulated delay-coordinate trajectories from the ground truth Lorenz-63 system (left) a model trained on noisy data using the pointwise objective (middle) and a model trained on noisy data using the measure-theoretic objective (right).}]{ \includegraphics[width = .8\textwidth]{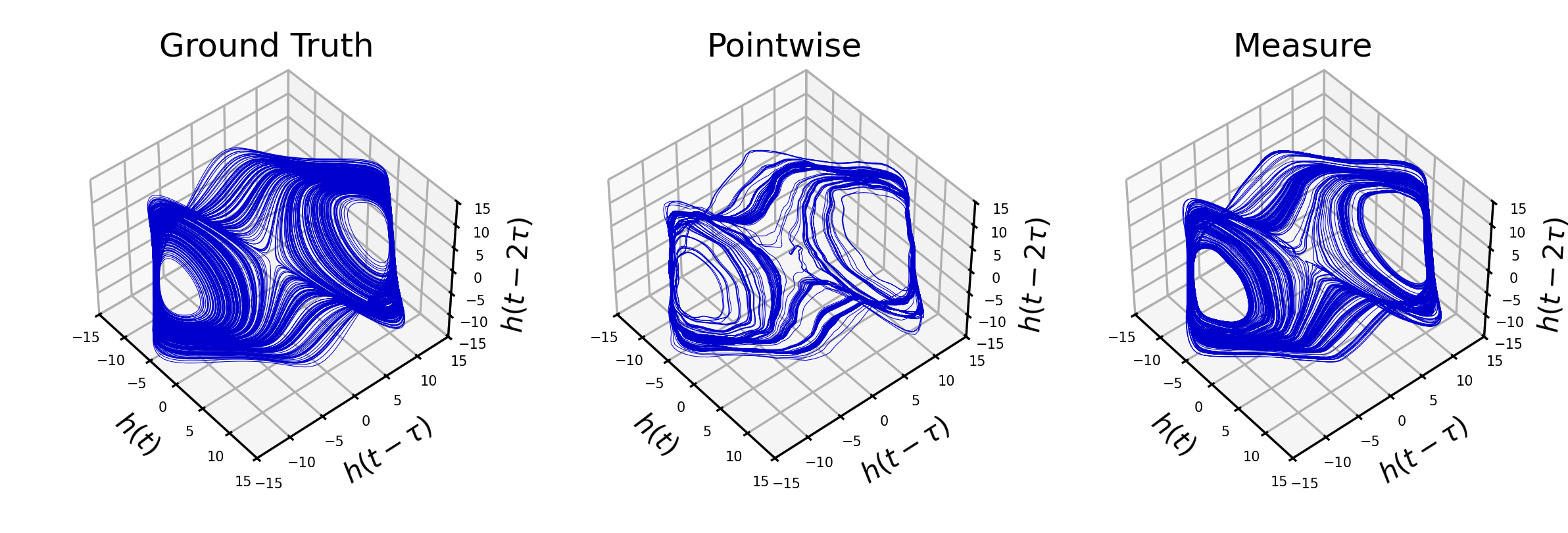}\label{fig:pred1}}\\
    \caption{\blue{Performing time-series prediction for the partially observed Lorenz-63 system from sparse and noisy data. Fig.~\ref{fig:pred3} shows the distribution of short-time prediction errors in delay coordinates and Fig.~\ref{fig:pred1} shows long trajectories of the learned delay-coordinate dynamics.}}
    \label{fig:lorenz_predict}
\end{figure}

 \blue{In Fig.~\ref{fig:lorenz_predict}, we visualize results for performing time-series prediction of the scalar-valued Lorenz-63 time-series $h_i = x_i\cdot e_1$, where $e_1\in \mathbb{R}^3$ is the first standard unit basis vector. After forming a long delay-coordinate trajectory in dimension $m = 4$ with $\tau = 0.2$, we assemble a sparse and noisy dataset by randomly sampling $N = 2\times 10^3$ initial conditions $\{\Phi(x_i)\}_{i=1}^N$ in the delay-coordinates. For each chosen initial condition, we observe its final position $\{\mathcal{T}(\Phi(x_i))\}_{i=1}^N$ in delay-coordinates after a time of $\Delta t = 0.1$ has elapsed. Thus, in this case, $\mathcal{T}$ corresponds to the time-$\Delta t$ (with $\Delta t = 0.1$) flow map of the delay-coordinate dynamics; see \eqref{eq:delay_dynamics}. Based on these input-ouput training pairs, we use a constrained \texttt{k}-means routine to partition the initial and final positions into $K = 10$ distinct empirical measures, i.e., we construct the measure-valued data $\{(\mathcal{T}\#(\Phi\#\mu_i), \Phi\#\mu_i)\}_{i=1}^K$. We then parameterize the delay-coordinate dynamics $\mathcal{T}_{\theta}$ as the time-$\Delta t$ flow map of a neural network-parameterized velocity with node sizes $4\rightarrow 100 \rightarrow 100 \rightarrow 100 \rightarrow 4$ and a hyperbolic tangent activation function. Throughout, all models are trained for $5\times 10^4$ steps using the Adam optimizer with an initial learning rate of $10^{-3}$ which decays on an exponential schedule to $10^{-5}$. We aim to learn $\mathcal{T}_{\theta}$ using both the pointwise and measure-theoretic objectives, as well as from clean data and noisy data. In the noisy case, the one-dimensional time series is corrupted by Gaussian noise with covariance $0.05$ before the delay coordinates are formed. For the measure-based learning \eqref{eq:distribution2} we choose $\mathfrak{D}$ as the MMD based on the kernel $k_1$.}

 \blue{The 1-step errors of the trained models, given by $|\mathcal{T}_{\theta}(\Phi(x_i)) - \mathcal{T}(\Phi(x_i))|^2$, are visualized in Fig.~\ref{fig:pred3} for the $\mathcal{T}_{\theta}$ trained under both the pointwise loss~\eqref{eq:pointwise2} and the measure-based loss~\eqref{eq:distribution2}. Fig.~\ref{fig:pred1} then displays simulated delay-coordinate trajectories of both the pointwise and measure-based models after training on the sparse and noisy dataset. While each objective only trains on a 1-step loss seeking to fit $\mathcal{T}_{\theta}$, we can integrate the learned vector field over arbitrary time horizons to generate longer trajectories. As shown in Fig.~\ref{fig:pred1}, the simulated trajectory from the measure-theoretic objective appears smoother and has better visual agreement with the ground truth trajectory.}

\begin{table}[h!]
    \centering
    \renewcommand{\arraystretch}{1.5}
      \begin{tabular}{|c|c|c|}
    \hline
       & Clean data & Noisy data \\
      \hline
     Pointwise & $(7.70  \pm 2.10) \times 10^{-4}$ & $(1.42 \pm 0.09)\times 10^{-1}$ \\
     Measure & $(2.66 \pm 0.26)\times 10^{-3}$ & $(4.83  \pm  0.27)\times 10^{-2}$ \\\hline
    \end{tabular}
    \caption{\blue{Comparing the prediction error between the pointwise \eqref{eq:pointwise2} and measure \eqref{eq:distribution2} approaches for forecasting the partially observed Lorenz-63 time-series. The error statistics are accumulated over 5 trials with different randomized neural network initializations.}}\label{tab:predict}
\end{table}

\subsection{NOAA Sea Surface Temperature Reconstruction}\label{subsec:realdata}

We now consider the problem of reconstructing the NOAA Sea Surface
Temperature (SST) from partial measurement data~\citep{AnImprovedInSituandSatelliteSSTAnalysisforClimate}. The SST dataset consists of weekly temperature measurements sampled at a geospatial resolution of $1\degree$; see Appendix~\ref{sec:exps} for a visualization. Our partial observation of the full SST dataset $\{\mathbf{z}(t_i)\}\subseteq \mathbb{R}^{44219}$ consists of the temperature time-series $\{x(t_i)\}\subseteq \mathbb{R}$ recorded at the location $(156 \degree, 40 \degree )$. Our goal is to use the delay state corresponding to $\{x(t_i)\}$ to learn a reconstruction map for the full state $\{\mathbf{z}(t_i)\}$. {That is, we aim to learn $\Phi^{-1}$ which maps $\{x(t_i)\}$ into  $\{\mathbf{z}(t_i))\}$. The success of the learned reconstruction map is then evaluated by applying it to unseen testing data in delay coordinates.} A similar problem is considered in~\citep{callaham2019robust,maulik2020probabilistic}. %

The dataset is partitioned into training and testing sets, where the training set contains 355 snapshots ($\approx 5$ years) and the testing set contains 1415 snapshots $(\approx 27$ years). Based on the time-series $\{x(t_i)\}$, we select a time delay of $\tau = 12$ (weeks) and an embedding dimension $m = 7$.  To reduce the computational cost in learning, similar to~\citep{maulik2020probabilistic,callaham2019robust}, we subtract off the temporal mean of the SST dataset and parameterize the full state of the system by the first $N_{\texttt{POD}}$ time-varying POD coefficients, $\{\alpha_{k}(t)\}_{k=1}^{N_{\texttt{POD}}}$, which are obtained via the method of snapshots; see \cite[Section 3.1]{maulik2020probabilistic}. That is, we perform the model reduction 
\begin{equation}\label{eq:POD}
    \mathbf{z}(t_i) - \overline{\mathbf{z}} = \sum_{k=1}^{N_{\texttt{POD}}} \alpha_k(t_i) \mathbf{m}_k ,\,\,\, \alpha_k(t_i) \in \mathbb{R}, \,\,\, \mathbf{m}_k\in \mathbb{R}^{44219},
\end{equation}
where $\overline{\mathbf{z}}\in \mathbb{R}^{44219}$ is the temporal mean of $\{ \mathbf{z}(t_i)\}$ and $\{\mathbf{m}_k\}_{k=1}^{N_{\texttt{POD}}}$ are the first $N_{\texttt{POD}}$ modes. We set $N_{\texttt{POD}} = 200$ and aim to learn the POD coefficient reconstruction map $\mathcal{R}_\theta:\mathbb{R}^{7}\to \mathbb{R}^{200}$ parameterized by $\theta$, given the paired data 
\begin{equation}\label{eq:POD_point}
    \left(x(t_i),x(t_i-\tau),\dots, x(t_i-6\tau)\right) \mapsto  \left(\alpha_1(t_i),\dots \alpha_{200}(t_i) \right).
\end{equation}

For this problem, all training samples are normalized via an affine transformation, such that the $L^{\infty}$ norm of the data vectors is at most 1.
We train the neural network parameterization $\mathcal{R}_\theta$ using both the pointwise (see Eqn.~\eqref{eq:pointwise}) and measure-based (see Eqn.~\eqref{eq:distribution}) approaches. The pointwise approach seeks to directly enforce the relationship~\eqref{eq:POD_point} through the mean-squared error, whereas the measure-based approach partitions the delay state into clusters and aims to push forward the empirical measure in each cluster into the corresponding measure in the POD space. We use a constrained \texttt{k}-means routine to evenly partition the delay state into $5$ clusters. A visualization of the 5 clusters can be found in Appendix~\ref{sec:exps}. We evaluate the performance of the learned $\mathcal{R}_{\theta}$ by {reconstructing} the time-varying POD coefficients $\{\alpha_k(t_i)\}$ for the testing set, which are then used to generate the full state $\{\mathbf{z}(t_i)\}$ according to~\eqref{eq:POD}. {More specifically, we reconstruct the $\{\alpha_k(t_i)\}$ by applying the trained model $\mathcal{R}_{\theta}$ to the delay state based on the time-series $\{x(t_i)\}.$}

We note that the challenge of learning the reconstruction map $\mathcal{R}_\theta:\mathbb{R}^7\to \mathbb{R}^{200}$ is exacerbated by the sparsity of the available training data. Specifically, we aim to learn a $200$-dimensional map using only $355$ training examples. Due to this data sparsity, we utilize MMD based on the Gaussian kernel $k_{2}(x,y) = \exp(-\|x-y\|_2^2/2\sigma^2)$ with $\sigma = 3$. The choice of a relatively large $\sigma$ acts as a form of regularization, helping to mitigate overfitting when dealing with sparse samples~\citep{feydy2020geometric}.

\begin{table}[h!]
\centering
\renewcommand{\arraystretch}{1.5}
      \begin{tabular}{|c|c|c|c|c|}
    \hline
      \backslashbox{Method}{MSE}& Initial & $10^3$ iters & $5\times 10^3$ iters & $2\times 10^4$ iters \\
      \hline
   Measure & $13.81 \pm 0.98$ & $0.72 \pm 0.01$ & $0.73 \pm 0.02$ & $0.80 \pm 0.02$ \\ 
   Pointwise & $13.95 \pm 1.39$ & $0.70 \pm 0.01$ & $0.82 \pm 0.01$ & $1.31 \pm 0.02$\\
    \hline
    \end{tabular}
    \vspace{.2cm}
    \caption{Reconstruction mean squared error (MSE) for the SST dataset after various numbers of training iterations. The measure-based approach (see Eqn.~\eqref{eq:distribution}) is less prone to overfitting than the pointwise approach (see Eqn.~\eqref{eq:pointwise}).}
    \label{tab:2}
\end{table}
   
In Fig.~\ref{fig:SST_RECON}, we visualize the pointwise and measure-based reconstructions of the SST example at testing weeks 425, 550, 675, and 800  after training both models for $25{,}000$ steps. One can observe that the measure-based results align more closely with the ground-truth snapshots, while the pointwise reconstructions exhibit numerous nonphysical oscillations.

\begin{figure}[h!]
  \centering
   \subfloat[SST Reconstruction at testing week 425]{ \includegraphics[width = .9\textwidth]{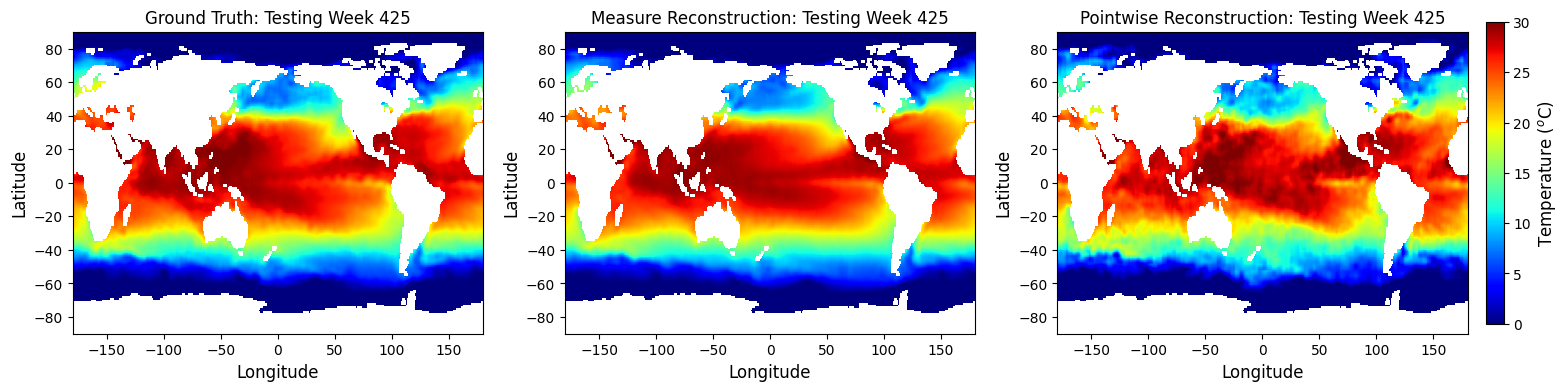}}\label{fig:5a}\\

   \subfloat[SST Reconstruction at testing week 550]{ \includegraphics[width = .9\textwidth]{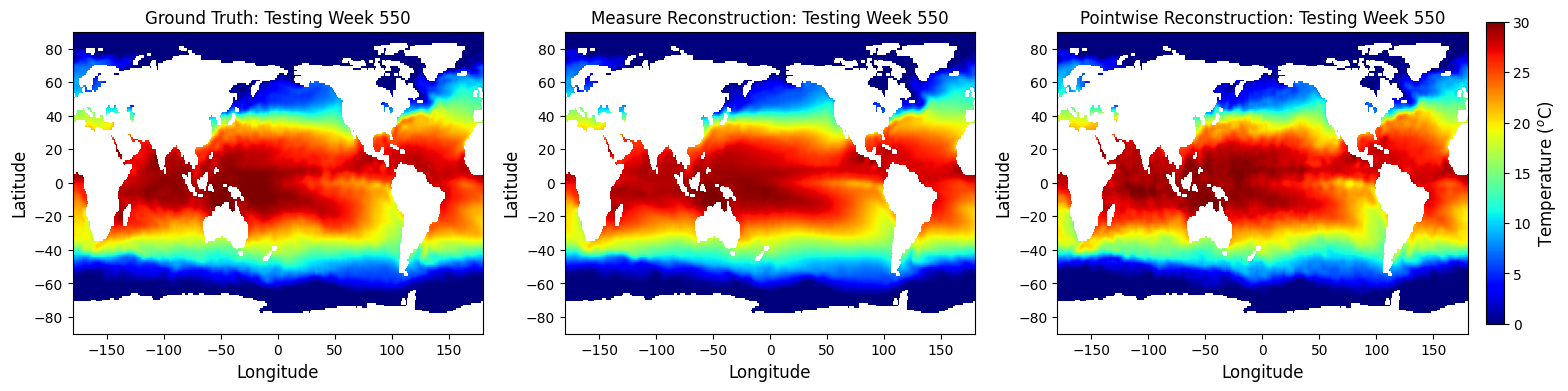}}\label{fig:5b}\\

   \subfloat[SST Reconstruction at testing week 675]{ \includegraphics[width = .9\textwidth]{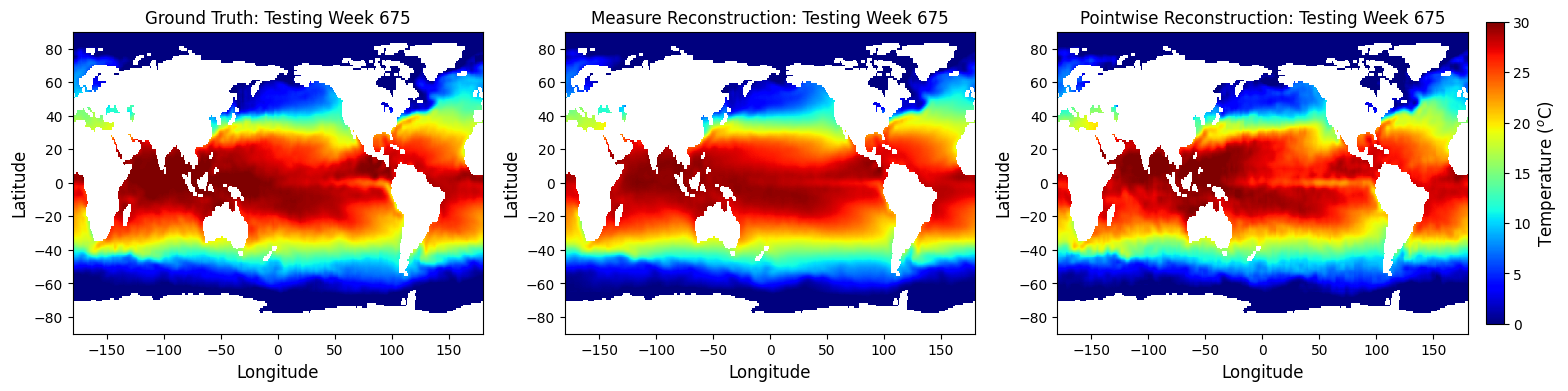}}\label{fig:5c}\\

    \subfloat[SST Reconstruction at testing week 800]{ \includegraphics[width = .9\textwidth]{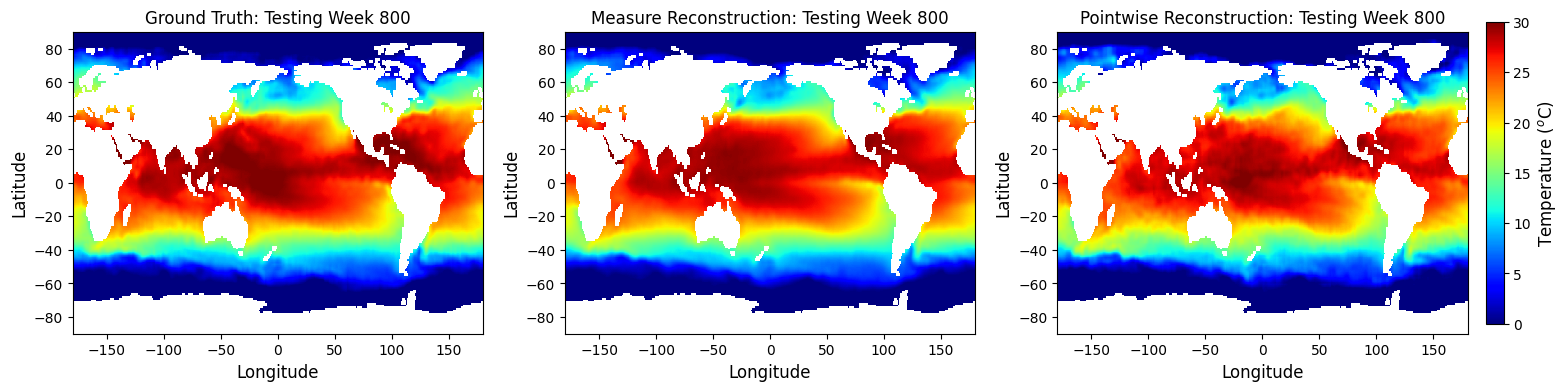}}\label{fig:5d}\\
    
    \caption{Visual comparison of the pointwise (see~\eqref{eq:pointwise}) and measure-based (see~\eqref{eq:distribution}) approaches to reconstructing the SST dataset at the testing weeks 425, 550, 675, and 800. The left column features the ground truth snapshot, the middle column shows the measure-based reconstruction, and the right column shows the pointwise reconstruction. }
    \label{fig:SST_RECON}
\end{figure}

\begin{table}
    \centering
      \begin{tabular}{|c|c|c|c|c|}
    \hline
      \diagbox{Method}{MSE} & Initial & $10^3$ iterations & $10^4$ iterations & $2 \times 10^4$ iterations \\
      \hline
      
    Gaussian MMD &  $158.15 \pm 0.27$ & $61.04 \pm 1.59$ & $61.22\pm 0.54$ & $67.91 \pm 0.49$ \\ 

    Energy MMD &  $158.15 \pm 0.27$ & $49.61 \pm 0.48$ & $67.85 \pm 0.41$ & $78.91 \pm 0.37$ \\ 
    
    Pointwise & $158.15 \pm 0.27$ & $51.00 \pm 0.19$ & $75.87 \pm 0.26$ & $84.85 \pm 0.25$\\
    \hline
    \end{tabular}
    \vspace{.2cm}
    \caption{Testing reconstruction mean squared error (MSE) for the ERA5 wind speed dataset after various iterations during training. For each loss function, the experiment was repeated $5$ times to approximate the mean and standard deviation for the reconstruction error.}
    \label{tab:3}
\end{table}

\begin{figure}
    \centering \includegraphics[width=.9\textwidth]{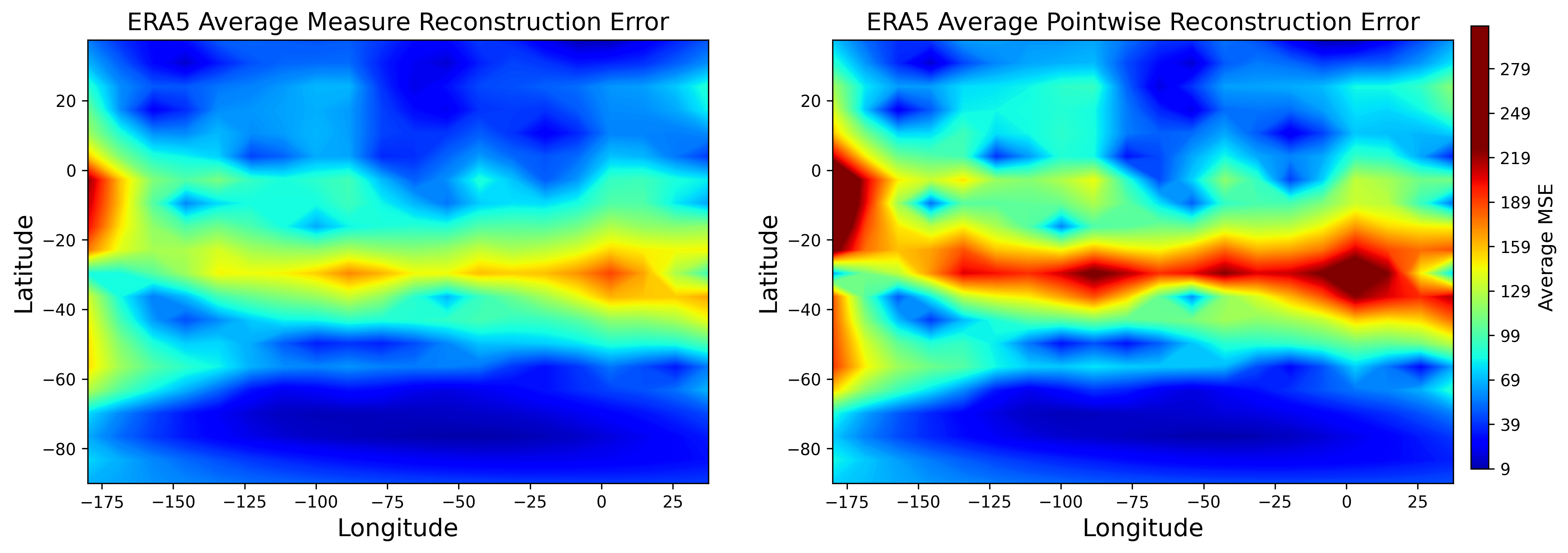}
    \caption{(Left) Spatial distribution of the average MSE for the measure-based reconstruction of the ERA5 dataset, based upon the Gaussian MMD. (Right) Spatial distribution of the average MSE for the pointwise reconstruction of the ERA5 dataset.}
    \label{fig:error_comparison}
\end{figure}

In Table~\ref{tab:2}, we compare the reconstruction errors of the pointwise and measure-based approaches after different numbers of training steps. Due to the sparsity of the training set, the pointwise approach is prone to overfitting, whereas the measure-based relaxation demonstrates greater robustness. After $2\times10^4$ training steps, the reconstruction error for the pointwise approach is approximately $1.6$ times larger than that of the measure-based approach. Both methods train a four-layer fully connected neural network with hyperbolic tangent activation, using the architecture $7 \rightarrow 100 \rightarrow 100 \rightarrow 100 \rightarrow 100 \rightarrow 200$ and a learning rate of $10^{-3}$. Each neural network  is trained $10$ times with different random initializations.

\subsection{ERA5 Wind Field Reconstruction}\label{subsec:ERA5}

Our final test reconstructs a portion of the ERA5 wind speed dataset using partial measurement data~\citep{https://doi.org/10.1002/qj.3803}. {Here, we will again aim to learn the full-state reconstruction map $\Phi^{-1}$ using training data from both the full state and a delay state corresponding to a partial observable. } We consider the ERA5 wind dataset sampled at a geopotential height of Z500, restricted to latitudes between $-180\degree$ and $-105\degree$ and longitudes between $-90\degree$ and $-15\degree$. The dataset is coarsened in both space and time, with measurements taken on approximately a $4\degree$ spatial grid with dimensions $20\times 20$, and separated by $6$ hours in time. For training, we use $2\times 10^3$ randomly sampled snapshots of both the $u$ and $v$ wind speed components (longitudinal and latitudinal, respectively) between 2000 and 2007. For testing, we consider the subsequent $5\times 10^3$ consecutive snapshots.

While the SST dataset studied in Section~\ref{subsec:realdata} exhibited strong periodic oscillations with one week between successive snapshots, the ERA5 wind speed dataset is sampled at a sub-daily rate and exhibits complex transient dynamics. Modeling these dynamics using partial observations from a single geospatial location, as done in Section~\ref{subsec:realdata}, is challenging. Similar to~\citep{maulik2020probabilistic,callaham2019robust}, we instead use a small number of randomly sampled sensors across the state space to collect partial measurement data. Unlike~\citep{maulik2020probabilistic,callaham2019robust}, we also consider time-lagged vectors corresponding to measurements at these spatial locations for performing the full-state reconstruction. This translates the problem of learning the full-state reconstruction map from a scalar-valued time-delayed observable into learning it from a vector-valued time-delayed observable. %

More specifically, the task involves learning the state reconstruction map $\mathbf{y}(t_i) \mapsto \mathbf{z}(t_i)$, where, for each fixed $t_i$, $\mathbf{y}(t_i) \in \mathbb{R}^{240}$ represents the time-delayed state across all observation locations, and $\mathbf{z}(t_i) \in \mathbb{R}^{800}$ is the full state of the system. We will now provide a detailed explanation of how the delayed state $\mathbf{y}(t_i)$ and the full state $\mathbf{z}(t_i)$ are constructed. The delay state is given by
\begin{equation}\label{eq:vec_delay}
    \mathbf{y}(t_i) = \left(\mathbf{x}(t_{i}),\mathbf{x}(t_{i}-\tau),\dots, \mathbf{x}(t_{i}-(m-1)\tau) \right)\in \mathbb{R}^{60m},
\end{equation}
where the vector $\mathbf{x}(t_i)\in \mathbb{R}^{60}$ represents the ensemble of partial observations at time $t_i$, i.e., 
\begin{equation}\label{eq:vec_partial}
   \mathbf{x}(t_i) = \left(u_{n_1}(t_i),\dots, u_{n_{30}}(t_i),v_{n_1}(t_i),\dots, v_{n_{30}}(t_i) \right) \in \mathbb{R}^{60}. 
\end{equation}
In \eqref{eq:vec_delay}, the time-delay and embedding dimension are heuristically chosen as $\tau = 6$ (hours) and $m = 4$. In \eqref{eq:vec_partial}, the indices $\{n_k\}_{k=1}^{30}$ correspond to 30 randomly sampled spatial locations at which we observe the longitudinal and latitudinal components of the wind speed. {While it was possible to obtain a reasonable model for the full state of the NOAA SST dataset from a single observable, the wind speed dynamics considered here exhibit significantly more complex spatio-temporal behavior. Thus, we find it necessary to use measurements from several spatial locations to perform accurate state reconstruction. The number 30 was also chosen heuristically.} The full longitudinal wind speed vector is $\mathbf{u}(t_i) = \left( u_1(t_i) , \dots, u_{400}(t_i) \right)\in \mathbb{R}^{400}$ representing values at the $20 \times 20$ spatial grid, 
while the full latitudinal wind speed vector is 
$\mathbf{v}(t_i)  = \left(v_1(t_i), \dots, v_{400}(t_i) \right)\in \mathbb{R}^{400}.$ The combined full state can then be written $\mathbf{z}(t_i) = (\mathbf{u}(t_i),\mathbf{v}(t_i))\in \mathbb{R}^{800}$, which is precisely the quantity we seek to predict from the delay state $\mathbf{y}(t_i)$; see \eqref{eq:vec_delay}.

We train models to approximate the state reconstruction map $\mathbf{y}(t_i)\mapsto \mathbf{z}(t_i)$ using both the pointwise reconstruction loss~\eqref{eq:pointwise} and the measure-based reconstruction loss~\eqref{eq:distribution}. Throughout, we normalize the training data by subtracting the temporal mean and ensuring that each feature has unit variance. Recall that we randomly sample $2\times 10^3$ times $t_i$ at which we observe both the delay state $\mathbf{y}(t_i)$ and the full state $\mathbf{z}(t_i)$, allowing us to form the paired data $\{(\mathbf{z}(t_i),\mathbf{y}(t_i))\}_{i=1}^{2000}.$ Thus, when learning the reconstruction map according to the pointwise loss~\eqref{eq:pointwise}, we directly enforce the relationship $\mathbf{y}(t_i) \mapsto \mathbf{z}(t_i)$, for $1\leq i\leq 2\times 10^3$, via the mean squared error.

In contrast, when learning the reconstruction map using the measure-based loss \eqref{eq:distribution}, we transform the paired data $\{(\mathbf{z}(t_i), \mathbf{y}(t_i))\}_{i=1}^{2000}$ into probability measures and aim to find a suitable pushforward map between the probability measures in the time-delayed state space and those in the full state space. We first apply a constrained \texttt{k}-means clustering algorithm to partition the time-delayed samples $\{\mathbf{y}(t_i)\}_{i=1}^{2000} \subseteq \mathbb{R}^{240}$ into 20 distinct clusters, each containing 100 samples in $\mathbb{R}^{240}$. In doing so, we obtain 20 pairs of probability distributions. During training, we then learn a model that pushes forward each measure in the time-delay coordinate system to the corresponding measure in the full state space. We evaluate the measure-based approach \eqref{eq:distribution} using both the Energy MMD loss (with the kernel {$k_{1}$}) and the Gaussian MMD loss (with a Gaussian kernel {$k_{2}$} and $\sigma = 25$).

We train all models using a fully connected neural network with hyperbolic tangent activation function and node sizes $240 \rightarrow 500 \rightarrow 500 \rightarrow 500 \rightarrow 800$. We use the Adam optimizer with a learning rate of $10^{-3}$. As shown in Table~\ref{tab:3}, the pointwise approach has a similar error to the Energy MMD and a lower error than the Gaussian MMD reconstruction after $10^3$ training iterations. However, as training continues, the pointwise approach overfits the data, whereas the measure-based approach using both MMD loss functions remains more robust. 

It is worth noting that choosing a relatively large bandwidth for the Gaussian MMD acts as additional regularization, helping to prevent overfitting to the locations of the training samples. For longer training times, the Gaussian MMD outperforms the Energy MMD in measure-based reconstruction. The reconstructions based on the Gaussian MMD are smoother than the pointwise reconstructions and more closely match the ground truth; a visualization can be found in Appendix~\ref{sec:exps}. In Fig.~\ref{fig:error_comparison}, we show the average spatial distribution of the reconstruction error for the Gaussian MMD and pointwise approaches.

\section{Conclusion}\label{sec:conclusions}
\blue{
This work presents a measure-theoretic variant of the classical Takens embedding theorem. While the original theory focuses on the embedding property of individual time-delayed trajectories, the proposed measure-theoretic generalization shifts the focus to probability distributions over the state space. Through Corollary~\ref{cor:1}, our main theoretical contribution, we have demonstrated that the embedding also holds within the space of probability distributions under the pushforward action of the time-delay map. This generalization is obtained by combining the classical Takens' theorem with geometric tools from the theory of optimal transport.
}

{
Building on these theoretical foundations, we devised a measure-theoretic computational routine for learning the inverse embedding map, also known as the full-state reconstruction map. This method enables the reconstruction of an entire high-dimensional dynamical system state from the time lags of a single observable. Our approach involves partitioning the observed trajectory in time-delay coordinates using a \texttt{k}-means clustering routine, where each cluster represents discrete probability measures. During training, we ensure that the corresponding discrete measures in the reconstruction space match the pushforward distributions of the measures in the time-delay coordinates. This training scheme represents a relaxation of classical pointwise matching, allowing for greater tolerance of pointwise errors in the final reconstruction based on the scale of each discrete measure in the reconstruction space.
}

{
While pointwise matching can be accurate under idealized, noise-free and densely sampled conditions, such settings are uncommon in practice. In scenarios with sparse or noisy observations, our distributional training can offer advantages. On synthetic tests and two geophysical datasets (NOAA sea-surface temperature and ERA5 wind speed) we observe stable reconstructions and competitive accuracy for the learned reconstruction map under challenging conditions.
}

{These results suggest several directions for further work, including integrating measure-theoretic tools with dynamical systems to model, predict, and control systems under uncertainty \citep{yang2023optimal,botvinick2023learning,botvinick2023generative,junge2024entropic,beier2024transfer}. They extend the scope of Takens-style embedding results and contribute to a measure-theoretic perspective on data-driven modeling. We will examine the stability of measure-based loss in future work and position it alongside established results for pointwise learning. This problem is nontrivial: stability depends jointly on the function approximation class, the magnitude and structure of noise, and the choice of distance or divergence on probability measures used in the optimization framework.
}

\bmhead{Acknowledgements}
J.~Botvinick-Greenhouse was supported by a fellowship award under contract FA9550-21-F-0003 through the National Defense Science and Engineering Graduate (NDSEG) Fellowship Program, sponsored by the Air Force Research Laboratory (AFRL), the Office of Naval Research (ONR) and the Army Research Office (ARO). M.~Oprea and Y.~Yang were partially supported by the National Science Foundation through grants DMS-2409855 and by the Office of Naval Research through grant N00014-24-1-2088. R.~Maulik was partially supported by U.S. Department of Energy grants, DE-FOA-0002493 and DE-FOA-0002905.

\section*{Declarations}
\begin{itemize}
\item Funding: J.~Botvinick-Greenhouse was supported by a fellowship award under contract FA9550-21-F-0003 through the National Defense Science and Engineering Graduate (NDSEG) Fellowship Program, sponsored by the Air Force Research Laboratory (AFRL), the Office of Naval Research (ONR) and the Army Research Office (ARO). M.~Oprea and Y.~Yang were partially supported by the National Science Foundation through grants DMS-2409855 and by the Office of Naval Research through grant N00014-24-1-2088. R.~Maulik was partially supported by U.S. Department of Energy grants, DE-FOA-0002493 and DE-FOA-0002905.

\item Competing interests: The authors have no competing interests to declare that are relevant to the content of this article.

\item Ethics approval and consent to participate: This does not apply to this article.
\item Consent for publication: All authors consent to publish this article.
\item Data availability: The data used in this article is available to download from the open-source repositories.
\item Materials availability: All materials related to this work are available to download from open-source domains.

\item Code availability: the code is available at \href{https://github.com/jrbotvinick/Measure-Theoretic-Time-Delay-Embedding}{https://github.com/jrbotvinick/Measure-Theoretic-Time-Delay-Embedding}.

\item Author contribution: J.~B.-G., M.~O., R.~M.~and Y.~Y.~designed the research, J.~B.-G.~and M.~O.~performed research; J.~B.-G.~, M.~O.~and  Y.~Y.~contributed new reagents or analytic tools; J.~B.-G.~and R.~M.~analyzed data; J.~B.-G., M.~O.~and Y.~Y.~wrote the paper; R.~M.~reviewed the paper.
\end{itemize}


\newpage

\begin{appendices}
In the appendices, we provide additional details to support our theoretical results in Section~\ref{sec:measure_embedding} of the main text, as well as our experimental results in Section~\ref{sec:experiments} of the main text. In Appendix~\ref{sec:proofs}, we provide proofs of Proposition \ref{prop:integral_not_changed}, Proposition \ref{prop:change_of_vars}, Lemma \ref{lem:diff_curve}, and Theorem \ref{thm:push_forward_differentiable}. These results are used in the main text to derive Theorem \ref{thm:main_embedding} and Corollary \ref{cor:1}, which are our paper's central theoretical results. In Section \ref{sec:exps}, we provide additional information and visualizations for our numerical experiments appearing in Section \ref{sec:experiments} of the main text.

\section{Proofs from the Main Text}\label{sec:proofs}

\begin{proof}[Proof of Proposition \ref{prop:integral_not_changed}]
    Let $v = P^\rho v + v^\perp = v^{\perp} + v^\top$ by the unique orthogonal decomposition, where $v^\perp \in \text{ker}(P^\rho)$ and $v^\top = P^\rho v \in T_\rho\mathcal{P}(M)$. Since $v^\perp \in \text{ker}(P^\rho)$, 
\begin{equation}\label{eq:proof1} \tag{S1}
    \nabla\cdot (v^\perp \rho) = 0 \implies \ \forall \varsigma \in \mathcal{D}(M): \quad \int_M g_x(\nabla \varsigma, v^\perp)\rd \rho(x) = 0. 
    \end{equation}
     Moreover, note that $\varphi \in \mathcal{D}(M)$, so~\eqref{eq:proof1} holds in particular for $\varsigma = \varphi$. 
     
    \begin{eqnarray*}
        \int_M g_x(\nabla\varphi(x), v)\rd \rho(x) &=& \int_M g_x(\nabla\varphi(x), v^\top + v^\perp)\rd \rho(x) \\
        &=& \int_M g_x(\nabla\varphi(x), P^\rho v)\rd \rho(x) +\int_M g_x(\nabla\varphi(x), v^\perp)\rd \rho(x)\,. 
    \end{eqnarray*}
    The last term in the equation above vanishes due to \eqref{eq:proof1}.
\end{proof}
\begin{proof}[Proof of Proposition \ref{prop:change_of_vars}]
Note that for any function $r: M \rightarrow \mathbb{R}$,  by definition $\nabla r$  is the vector field in $TM$ such that $g(\nabla r , X) = dr(X)$, $\forall X \in TM$. We apply this to $r = \varphi \circ f $:
$$
g_x(\nabla(\varphi \circ f)_x, X_x) = d(\varphi \circ f) (X_x) = d\varphi_{f(x)} (df_x X_x) = d\varphi_{f(x)} Y_{f(x)}\,,
$$
where $Y_{f(x)} := df_x X_x$ is a vector field on $N$. Moreover, since $d\varphi_{f(x)}$ lies in the cotangent bundle on $N$, we have $d\varphi_{y} Y_{y} = q_y(\nabla \varphi_y, Y_y)$. %
By plugging this into the equation above with $y = f(x)$, we obtain~\eqref{eq:change_of_vars}.
\end{proof}

\begin{proof}[Proof of Lemma \ref{lem:diff_curve}]
Consider that $\rho_t$ is an AC curve. Proposition~\ref{prop:vectors} guarantees the existence of a vector field $v_t\in L^2(M, \rho_t)$ along $\rho_t$ such that the continuity equation holds. However, $v_t$ is not necessarily in the tangent space $T_{\rho_t} \mathcal{P}(M) \subset L^2(M, \rho_t)$. For an AC curve $\rho_t$, we instead define $\frac{d}{dt}\rho_t = P^{\rho_t} v_t \in T_{\rho_t}\mathcal{P}(M)$. By the definition of the projection, $\left(\rho_t, \frac{d}{dt}\rho_t\right)$ satisfies the continuity equation. Thus, $\frac{d}{dt}\rho_t \in T_{\rho_t}\mathcal{P}(M)$ is a tangent vector field along the curve $\rho_t$ which proves our result. 
\end{proof}
    \begin{proof}[Proof of Theorem \ref{thm:push_forward_differentiable}]  We proceed in two steps. First, let $\rho_t$ be an AC curve, and $v_t = \frac{d}{dt}\rho_t$ whose existence is guaranteed by Proposition~\ref{prop:vectors}. We then prove that $(F(\rho_t), \widetilde{dF}_{\rho_t}(v_t))$ satisfies the continuity equation and that $\int_0^1 \|\widetilde{dF}_{\rho_t}(v_t)\|_{L^2(F(\rho_t))} \rd t< \infty$, which guarantees that $F$ is AC by Proposition~\ref{prop:vectors}. Secondly, we extend the previous results to the projected version $dF_\rho = P^{F(\rho)}\widetilde{dF_{\rho}}(v)$ which lies in $T_{F(\rho)}\mathcal{P}(N)$ by the definition of the projection operator (see~\eqref{eq:projection}).
        
        For the first part, let $\rho_t$ be an AC curve in $\mathcal{P}(M)$. Then
        \begin{align*}
            \int_0^1 \|\widetilde{dF}_{\rho_t}(v_t)\|_{L^2(F(\rho_t))}\, \rd t &= \int_0^1 \sqrt{\int_N q_y(\widetilde{d F}_{\rho_t}(v_t)(y)\,,\,\widetilde{d F}_{\rho_t}(v_t)(y) )\,  \rd \nu_t(y)}\,\rd t\,, \quad \text{where }\nu_t = F(\rho_t)  \\
            &= \int_0^1 \sqrt{\int_N q_y\big(df_{f^{-1}(y)} (v_t(f^{-1}(y)))\,, \, df_{f^{-1}(y)} (v_t(f^{-1}(y)))\big) \, \rd(f\#\rho_t)(y)}\,\rd t\\
            & = \int_0^1\sqrt{\int_M q_{f(x)}(df_x(v_t(x)), df_x(v_t(x)))\, \rd\rho_t(x)}\,\rd t\\
            & \leq \int_0^1 \sqrt{\int_M \left(\sup_{x \in M }\|df_x\|^2 \right) g_x(v_t(x), v_t(x)) \,\rd\rho_t(x)}\,\rd t \qquad \label{eq:inequality} \tag{S2}\\ 
           & = \sup_{x \in M } \|df_x\| \int_0^1 \|v_t\|_{L^2(\rho_t)}\,\rd t\,.
        \end{align*}
        Since $\sup_{x \in M}\|df_x\| < \infty$, the last term is bounded  as a result of Proposition \ref{prop:vectors}. The inequality~[\ref{eq:inequality}] comes from the definition of the norm $\|df_x\|$ and of the supremum:
           \begin{align*}
             \|df_x\|^2 = \sup_{w \in T_x M} \frac{q_{f(x)}(df_x w, df_x w)}{g_x(w, w)} \implies \\
              \int_M q_{f(x)}(df_x( v_t(x)), df_x( v_t(x)))\, \rd \rho_t(x) &\leq \int_M \|df_x\|^2 g_x(v_t(x), v_t(x)) \, \rd \rho_t(x) \\ &\leq \sup_{x \in M }\|df_x\|^2 \int_M g_x(v_t(x), v_t(x)) \, \rd\rho_t(x)\,.
        \end{align*}
        
        Next, we show that $(F(\rho_t), \widetilde{dF}_{\rho_t}(v_t))$  satisfies the continuity equation~\eqref{eq:continuity_dist}. For any test function $\varphi\in \mathcal{D}(N\times [0,T])$,
        \begin{align}
        &\int_0^1   \int_N \left(\frac{\partial \varphi}{\partial t}(y) + q_{y}\left(\nabla \varphi (y, t),  \widetilde{dF}_{\rho_t}(v_t) (y) \right) \right) \rd \nu_t (y) \rd t\,,\quad \nu_t = F(\rho_t) \label{eq:continuity_N}\tag{S3} \\
           = & \int_0^1 \int_N \left(\frac{\partial \varphi}{\partial t}(y,t) + q_{y}(\nabla \varphi (y, t),  df_{f^{-1}(y)}\left(v_t\left(f^{-1}(y)\right)\right) \right) \rd(f\#\rho_t) (y) \rd t  \nonumber \\
          = &\int_0^1 \int_M \left(\frac{\partial \varphi }{\partial t} (f(x),t) + q_{f(x)}(\nabla \varphi (f(x), t), df_x (v_t(x))) \right) \rd \rho_t(x) \rd t  \nonumber \\
          = & \int_0^1 \int_M \left( \frac{\partial  \widetilde{\varphi}}{\partial t}(x,t) + g_x\left(\nabla \widetilde{\varphi}(x,t), v_t(x)\right) \right) \rd\rho_t(x) \rd t\,,\quad \widetilde{\varphi}(x,t):=\varphi(f(x),t)\,. \label{eq:continuity_M} \tag{S4}
        \end{align}
        We used Proposition~\ref{prop:change_of_vars} to obtain~\eqref{eq:continuity_M}. 
        Moreover, note that~\eqref{eq:continuity_M} is exactly the weak formulation of the continuity equation for $(\rho_t, v_t)$ where the test function is $\widetilde{\varphi}$. We now check if $\widetilde{\varphi} \in \mathcal{D}(M\times [0,T])$. Since $f$ is continuously differentiable and $\varphi \in \mathcal{D} (N\times [0,T])$, we have $\widetilde{\varphi}\in \mathcal{C}^1(M\times [0,T])$. The support of $\widetilde{\varphi}$ is
        $$\text{supp}(\widetilde{\varphi}) = \overline{\{(x,t): \varphi((f(x),t)) \neq 0 \}} \subseteq \overline{ (f\otimes \text{\textit{Id}})^{-1}\left( \text{supp}(\varphi)\right)}\,.
       $$
        Since  $\text{supp}(\varphi)$ is compact and $(f\otimes \text{\textit{Id}})$ is proper, $(f\otimes \text{\textit{Id}})^{-1}\left(\text{supp}(\varphi)\right)$ is also compact. Moreover, supp$(\widetilde{\varphi})$ is a closed subset of a compact set, and hence itself compact. Thus, $\widetilde{\varphi} \in \mathcal{D}(M\times [0,T])$. Since  $(\rho_t, v_t)$ satisfies~\eqref{eq:continuity} by assumption, \eqref{eq:continuity_M} is always zero, which implies that $(F(\rho_t), \widetilde{dF}_{\rho_t}(v_t))$ satisfies the continuity equation~\eqref{eq:continuity_dist}.
        Finally, we turn to $dF_{\rho_t}(v_t)$.
        \begin{align*}
           & \int_0^1   \int_N \left(\frac{\partial \varphi}{\partial t}(y) + q_{y}\left(\nabla \varphi (y, t),  dF_{\rho_t}(v_t) (y)\right) \right) \rd \nu_t(y) \rd t
           \,, \quad \nu_t = F(\rho_t)\\
        = & \int_0^1 \int_N \frac{\partial \varphi}{\partial t}(y) \rd \nu_t (y)  \rd t 
             + \int_0^1 \int_N  q_{y}\left(\nabla \varphi (y, t),  P^{F(\rho_t)}\widetilde{dF}_{\rho_t}(v_t) (y)\right)\rd \nu_t (y)  \rd t \\
             = &
              \int_0^1 \int_N \frac{\partial \varphi}{\partial t}(y) \rd \nu_t (y)  \rd t  + \int_0^1 \int_N q_{y}\big(\nabla \varphi (y, t), \widetilde{dF}_{\rho_t}(v_t) (y)\big)\rd \nu_t (y) \rd t = 0\,,
        \end{align*}
        where we used Proposition~\ref{prop:integral_not_changed}. Thus, $dF_{\rho_t}(v_t)$ also satisfies the continuity equation. 
    \end{proof}

\section{Numerical Experiments}\label{sec:exps}
\subsection{Synthetic Examples}
The equations for the dynamical systems studied in Section \ref{subsec:noisydata} are given by 
$$\label{eq:systems}
 \underbrace{\begin{array}{l}
\dot{x}_1 = a_1(x_2-x_1) \\
   \dot{x}_2 = x_1(a_2 - x_3) - x_2 \\
   \dot{x}_3 = x_1x_2 - a_3 x_3 
\end{array}}_{\text{Lorenz-63}}\,,
 \qquad   \underbrace{\begin{array}{l}
   \dot{x}_1 = -x_2-x_3 \\
   \dot{x}_2 = x_1+ b_1 x_2 \\
   \dot{x}_3 = b_2+x_3(x-b_3)\end{array}}_{\text{R\"ossler}}\,, \qquad
\underbrace{\dot{x}_i = r_ix_i\Bigg( 1 - \sum_{j=1}^N \alpha_{ij} x_j\Bigg)}_{\text{Lotka--Volterra}}.
$$
Their parameters are respectively
$$
a = \begin{bmatrix}
    10\\
    28\\
    8/3
\end{bmatrix}\,,\qquad b = \begin{bmatrix}
    0.1\\
    0.1\\
    14
\end{bmatrix} \,,\qquad r = \begin{bmatrix}
    1\\
    0.72\\
    1.53\\
    1.27
\end{bmatrix}\,,  \qquad \alpha = \begin{bmatrix}
    1 & 1.09 & 1.52 & 0 \\
    0 & 1 & 0.44 & 1.36\\
    2.33 & 0 & 1 & 0.47\\
    1.21 & 0.51 & 0.35 & 1
\end{bmatrix},$$
which are standard values known to produce chaotic trajectories.  In Section \ref{subsec:noisydata} of the main text, we compared the pointwise (see~\eqref{eq:pointwise}) and measure-based (see~\eqref{eq:distribution}) reconstruction schemes for sparse and noisy data coming from these systems. 

 Here, we show another numerical example, but this time, a large amount of clean, noise-free data is used instead. Fig.~\ref{fig:prob_embedding} shows the results in which we perform the full state reconstruction of the same three systems. For this test, we partition the delay state into $100$ cells, which results in $100$ different probability measures for the measure-based loss (see~\eqref{eq:distribution}). It is worth noting that each measure is an empirical distribution based on thousands of samples. To reduce the computational cost, we used mini-batching to reduce the number of samples in each of the $100$ empirical distributions from $5000$ to $50$.
We remark that mini-batch training is not performed in the tests shown in the main text since all examples there are on sparse datasets where the problem instead becomes a lack of data. 

After subdividing the training data,  we then parameterize $\mathcal{R}_{\theta}$ as a two-layer feed-forward neural network with $100$ nodes in each layer and a hyperbolic tangent activation function. We enforce the distributional loss~\eqref{eq:distribution} during training. We use a learning rate of $10^{-3}$. Despite using only $K = 100$ measures for learning $\mathcal{R}_{\theta}$, after $100$ epochs of training, we find that the network {reconstructs} the full state of each dynamical system with high-precision. See Fig.~\ref{fig:2a} for the Lorenz-63 system,  Fig.~\ref{fig:2b} for the R\"ossler system and Fig.~\ref{fig:2c} for the Lotka--Volterra system.
\begin{figure}
    \centering
   \subfloat[(Left) The delay state for the Lorenz-63 system based on the time-series $x(t)$ with $\tau = 0.18$ and $d = 4$. (Right) {Reconstructing} the full state $(x(t),y(t),z(t))$ from the time series $x(t)$.]{ \includegraphics[width = .9\textwidth]{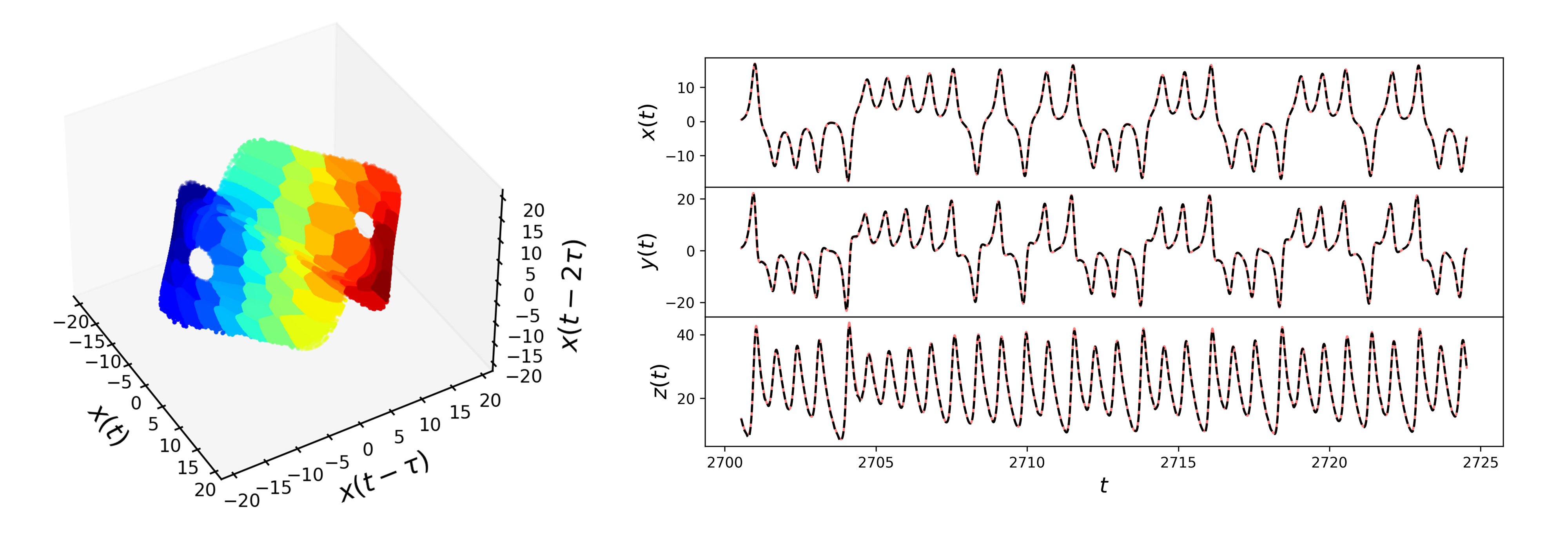}\label{fig:2a}}\\
    \subfloat[(Left) The delay state for the R\"ossler system based on the time-series $x(t)$ with $\tau = 1.44$ and $d = 4$. (Right) {Reconstructing} the full state $(x(t),y(t),z(t))$ from the time series $x(t)$.]{ \includegraphics[width = .9\textwidth]{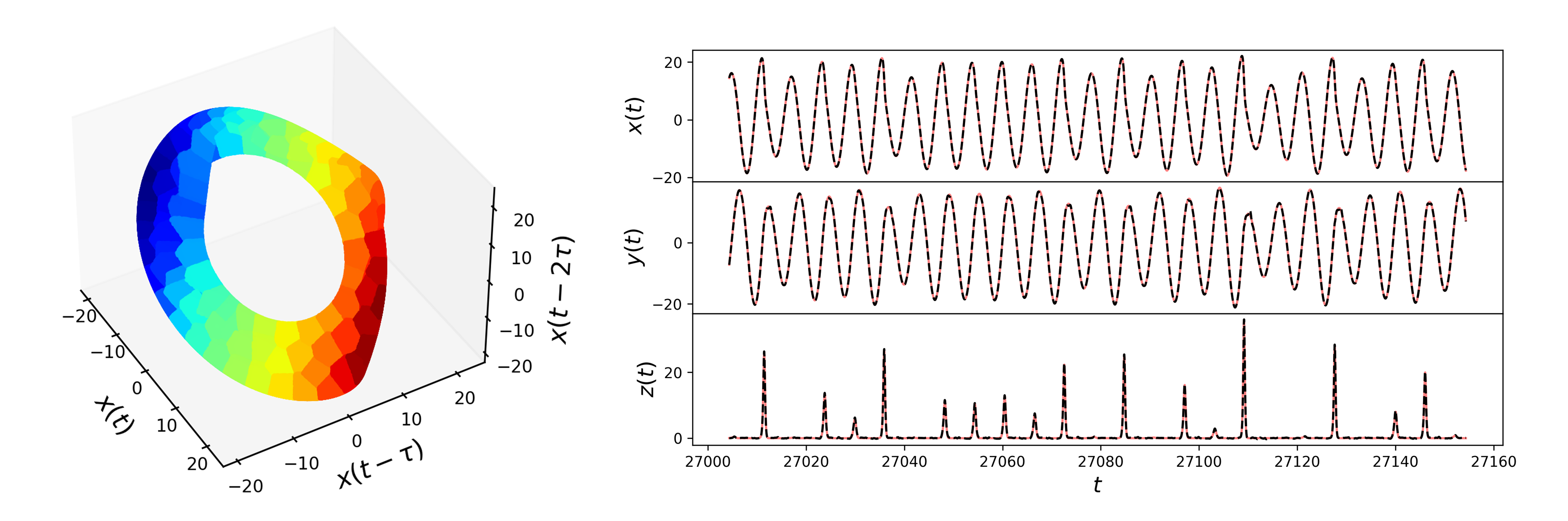}\label{fig:2b}}\\
       \subfloat[(Left) The delay state for the Lotka--Volterra system based on the time-series $x_1(t)$ with $\tau = 6.90$ and $d = 5$. (Right) {Reconstructing} the full state $(x_1(t),x_2(t),x_3(t),x_4(t))$ from $x_1(t).$]{ \includegraphics[width = .9\textwidth]{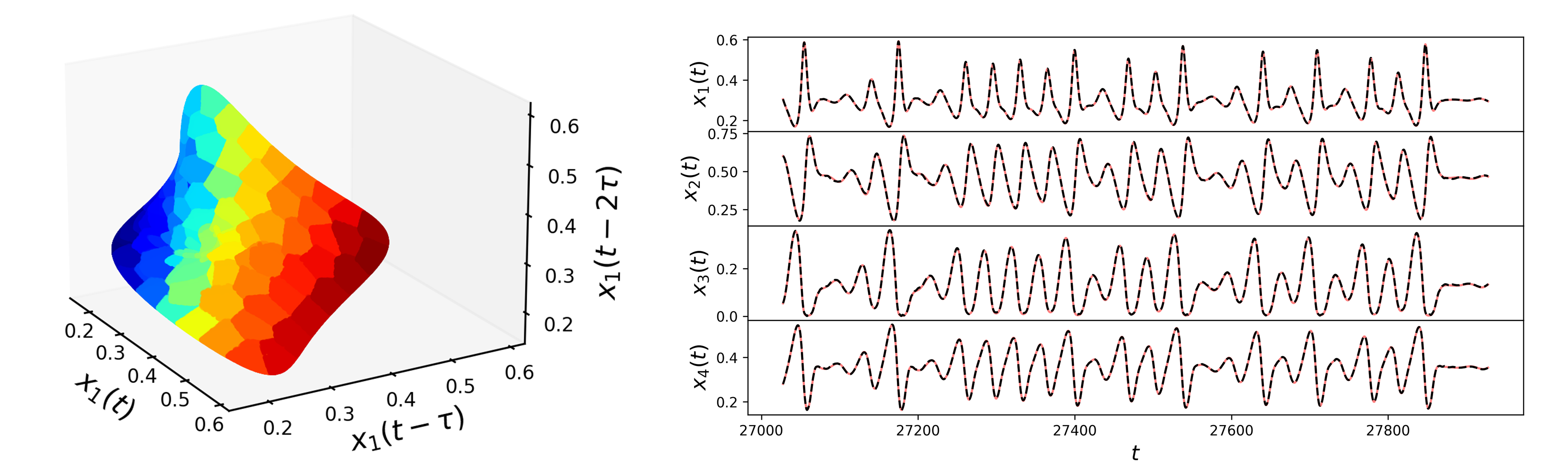}\label{fig:2c}}\\
    \caption{Learning the full-state reconstruction map for three low-dimensional chaotic attractors. The measures $\Phi\# \mu_i$ on the delayed attractor are shown in the left column with each measure colored differently. On the right column, the neural network {reconstruction} during testing is shown as the dotted-black line, whereas the ground truth is in red. In this case, the data is noise-free.}
\end{figure}
\subsection{NOAA SST Reconstruction}
In Section \ref{subsec:realdata}, we performed full state reconstruction on the NOAA SST dataset using both the pointwise and measure-based approaches. Fig.~\ref{fig:SSTData} visualizes the dataset we use to perform the reconstruction. Fig.~\ref{fig:4a} shows a single snapshot of the dataset, as well as the geospatial location at which we collect partial observational data. Fig.~\ref{fig:4b} plots the temperature time series at this location. Figures \ref{fig:4c} and \ref{fig:4d} visualize projections of the input-output measure data used to train the model according to \eqref{eq:distribution}. More specifically, Fig.~\ref{fig:4d} plots the input measures constructed according to the partially observed data in time-delay coordinates, while \ref{fig:4c} shows the output measures in the POD reconstruction space. 
\begin{figure}[h!]
    \centering
   \subfloat[NOAA SST dataset snapshot]{ \includegraphics[width = .49\textwidth]{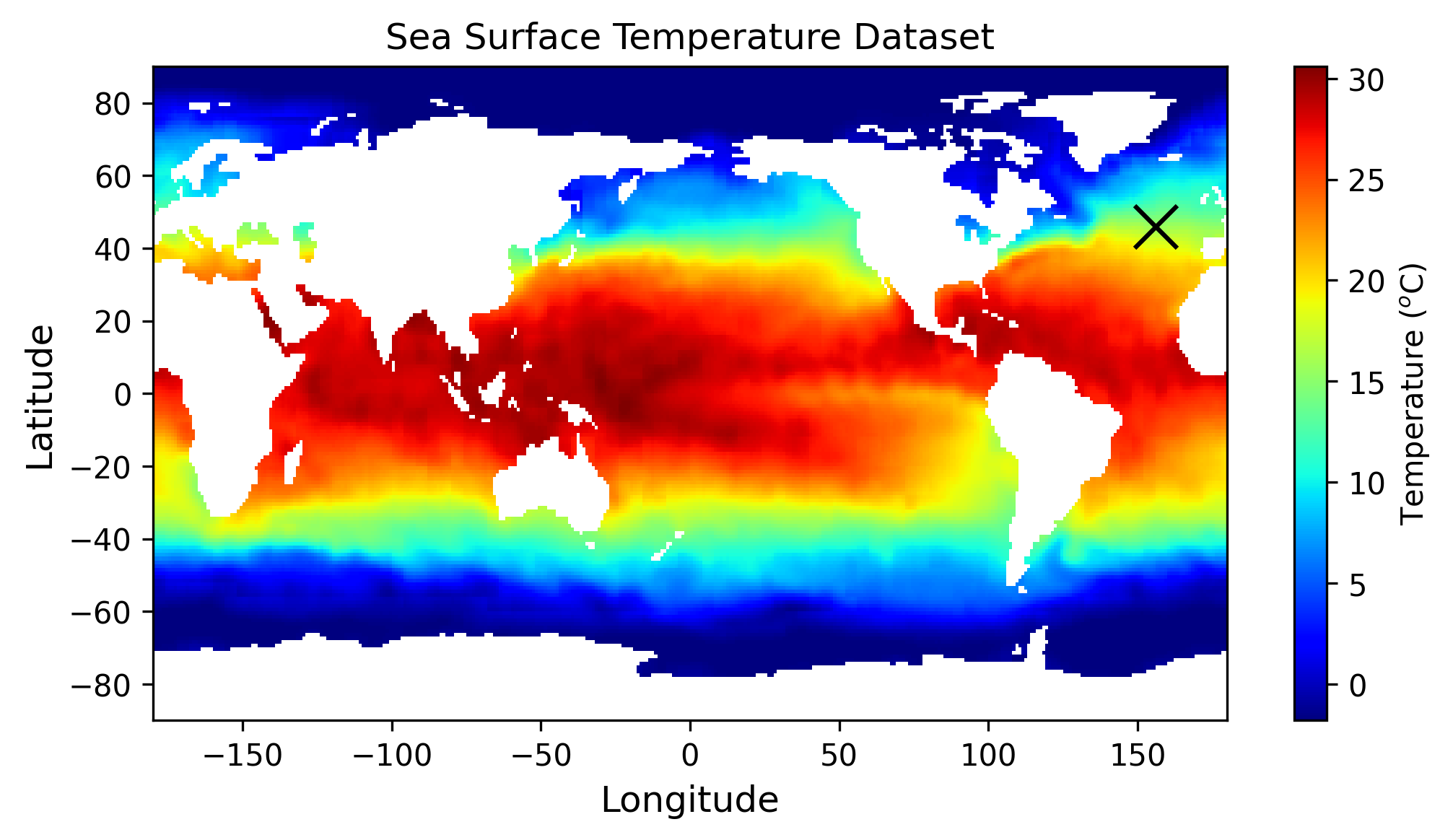} \label{fig:4a}}
    \subfloat[SST observation at $(156 \degree, 40 \degree )$]{ \includegraphics[width = .45\textwidth]{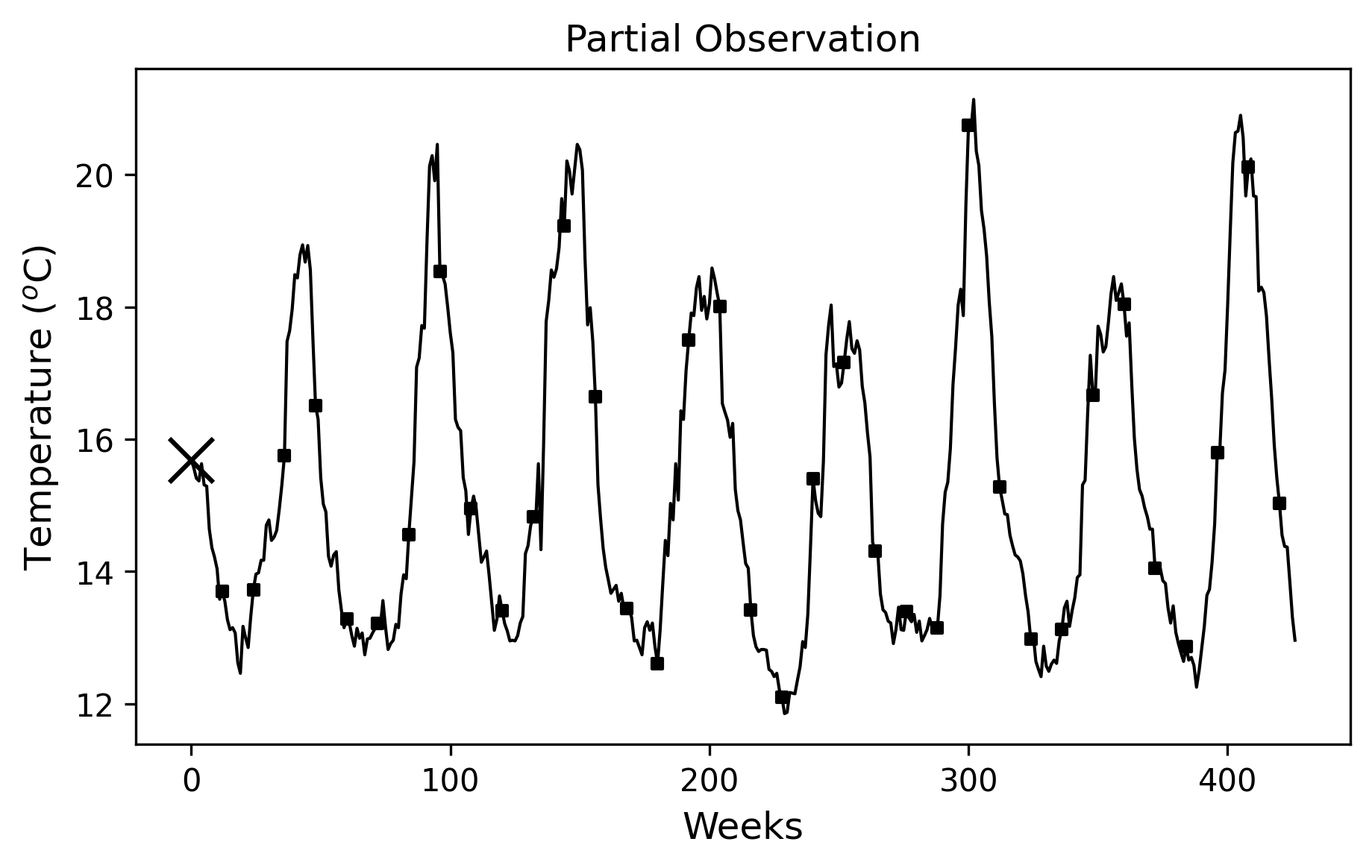}\label{fig:4b}}\\
    \subfloat[Training data in the POD space]{ \includegraphics[width = .35\textwidth]{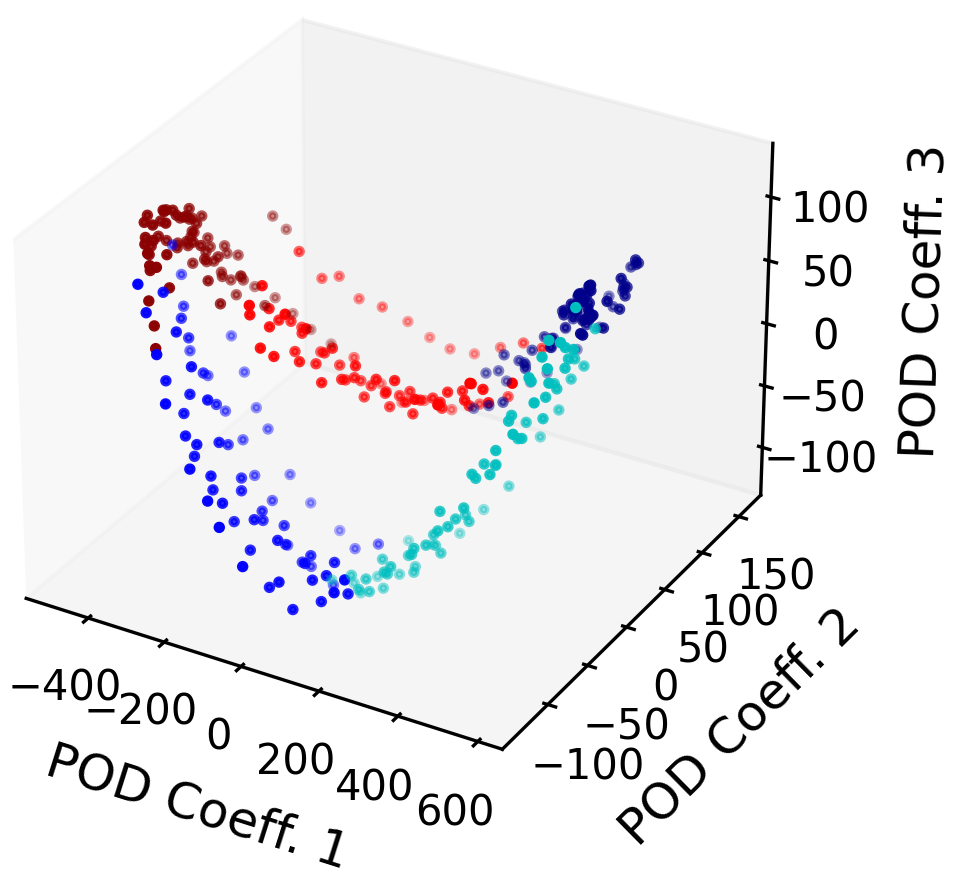}\label{fig:4c}}
    \hspace{1cm}
    \subfloat[Time-delayed observation]{ \includegraphics[width = .35\textwidth]{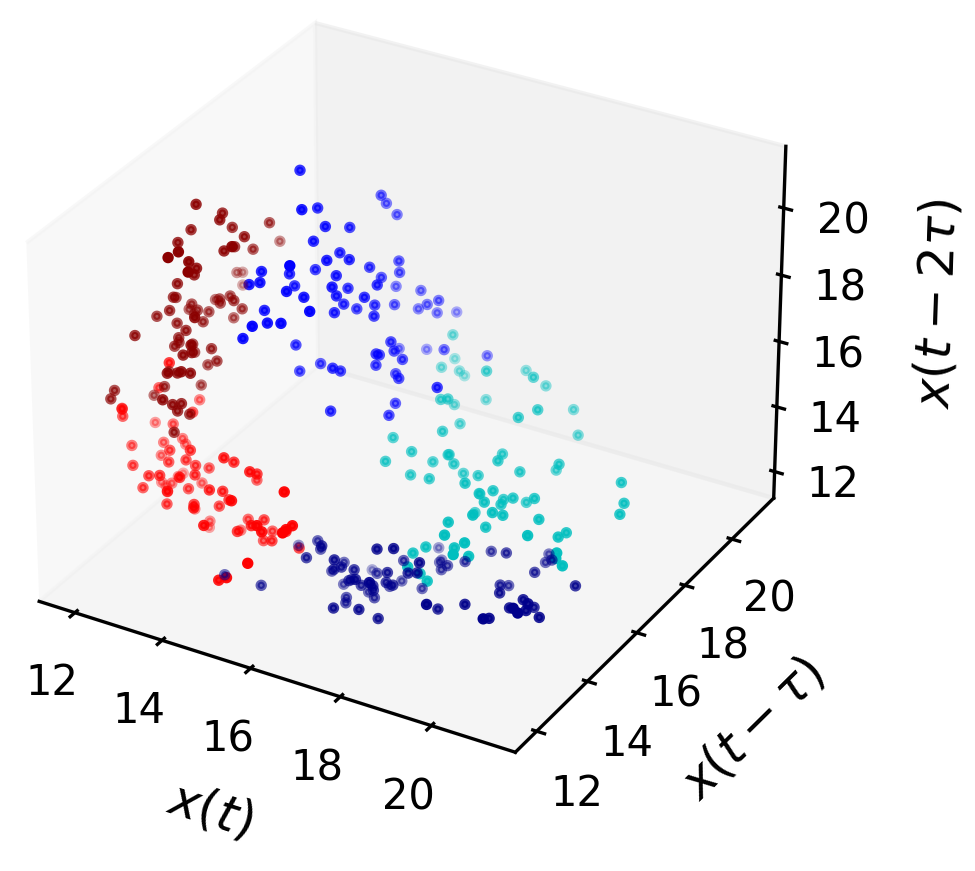}\label{fig:4d}}
    \caption{(a) A single snapshot from the NOAA SST dataset. (b) The time-series from the SST dataset sampled at $(156 \degree, 40 \degree )$, which we regard as our partial observation of the full state. The squares illustrate the time increment of $\tau = 12$ which is used to form the delay coordinates. (c) A three-dimensional projection of the POD coefficients $\{\alpha_k(t_i)\}$ which parameterize the full state. (d) A three-dimensional projection of the delayed time-series in (b). The colors in (c) and (d) reflect the five discrete measures which are used to train our measure-based model. }
    \label{fig:SSTData}
\end{figure}

\subsection{ERA5 Wind Field Reconstruction}
In Section \ref{subsec:ERA5}, we performed full state reconstruction on a portion of the ERA5 wind field dataset. Fig.~\ref{fig:ERA5probes} visualizes a single snapshot of the dataset, including the randomly sampled geospatial locations at which we collect partial observational data, as well as an example wind speed time-series at one of these locations. Fig.~\ref{fig:ERA5_RECON} visualizes results for performing the full-state reconstruction on the dataset using both the pointwise (see~\eqref{eq:pointwise}) and measure-based (see~\eqref{eq:distribution}) approaches. 

\begin{figure}[h!]
    \centering
 \subfloat[ERA5 wind speed snapshot]{ \includegraphics[width = .45\textwidth]{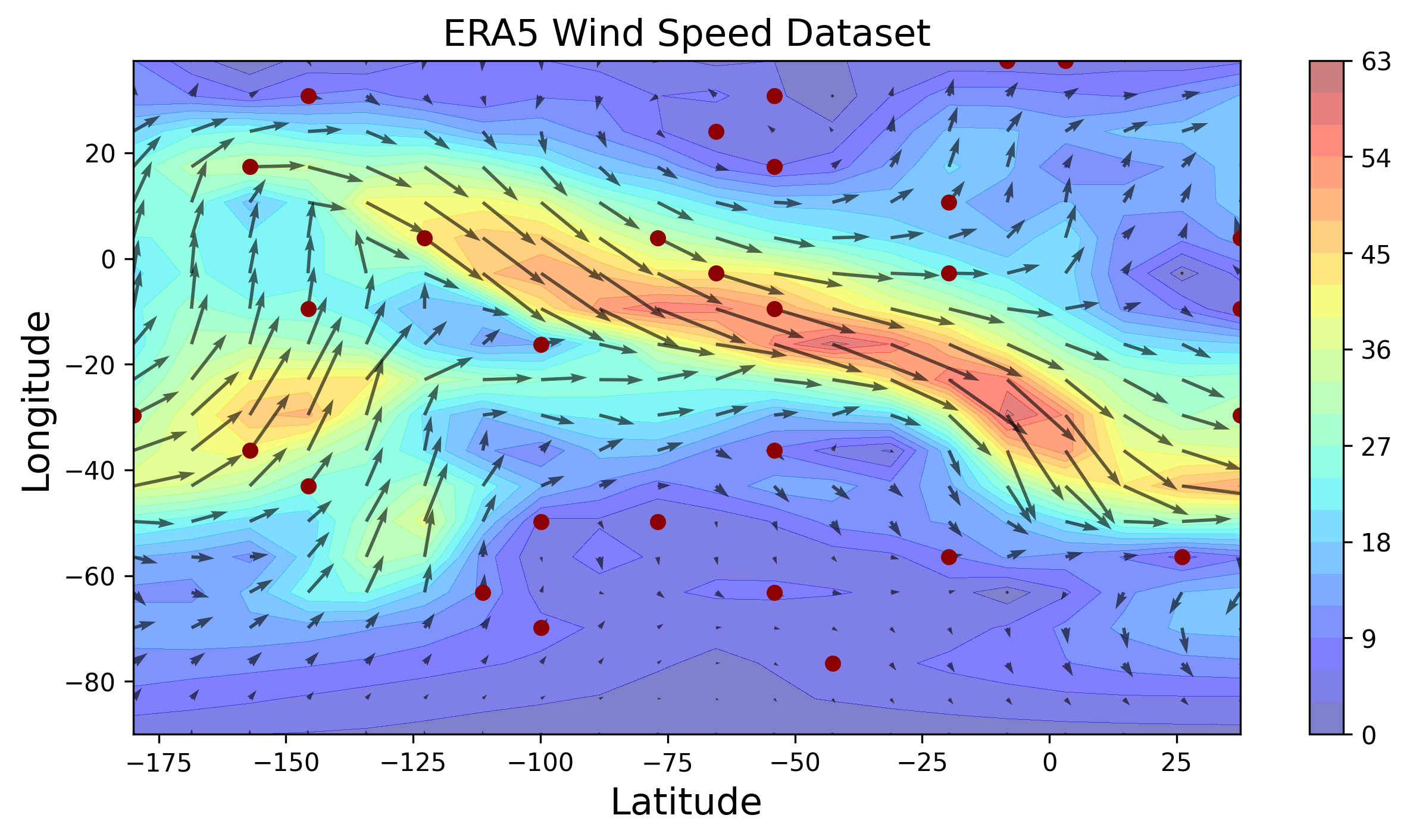}}\label{fig:6a}
   \subfloat[Example trajectory from ERA5]{ \includegraphics[width = .45\textwidth]{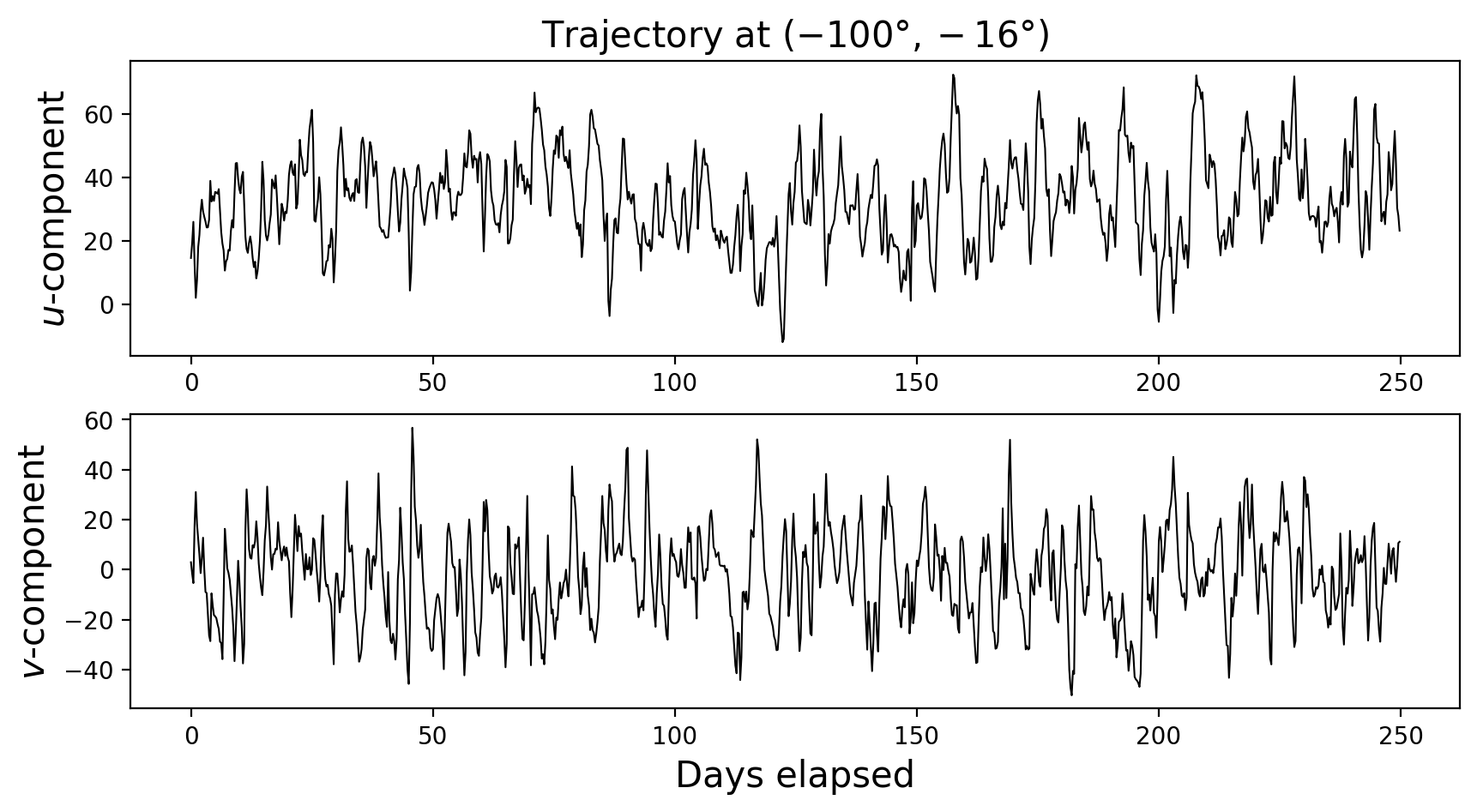}}\label{fig:6b}\\

    \caption{Visualization of the ERA5 wind speed dataset. (Left) A single snapshot in time of the dataset. The vector field indicates the direction of the wind, the contours illustrate the magnitude of the wind speed, and the 30 red circles show the locations at which we obtain partial observations of the wind field. (Right) The time trajectory corresponding to the geospatial location $(-100\degree,-16\degree)$. }
    \label{fig:ERA5probes}
\end{figure}
\begin{figure}
  \centering
   \subfloat[Wind field reconstruction at testing week 0]{ \includegraphics[width = \textwidth]{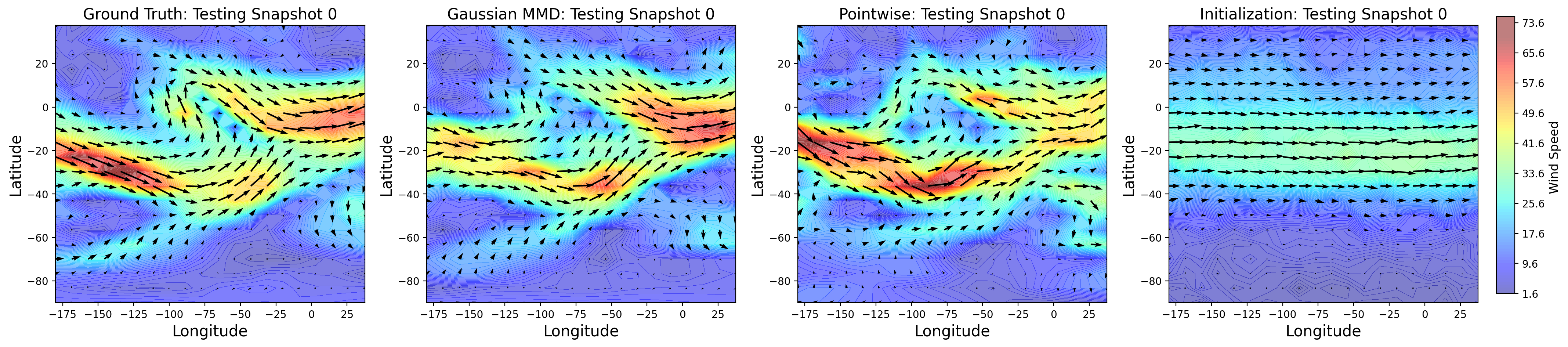}}\\

   \subfloat[Wind field reconstruction at testing week 1500]{ \includegraphics[width = \textwidth]{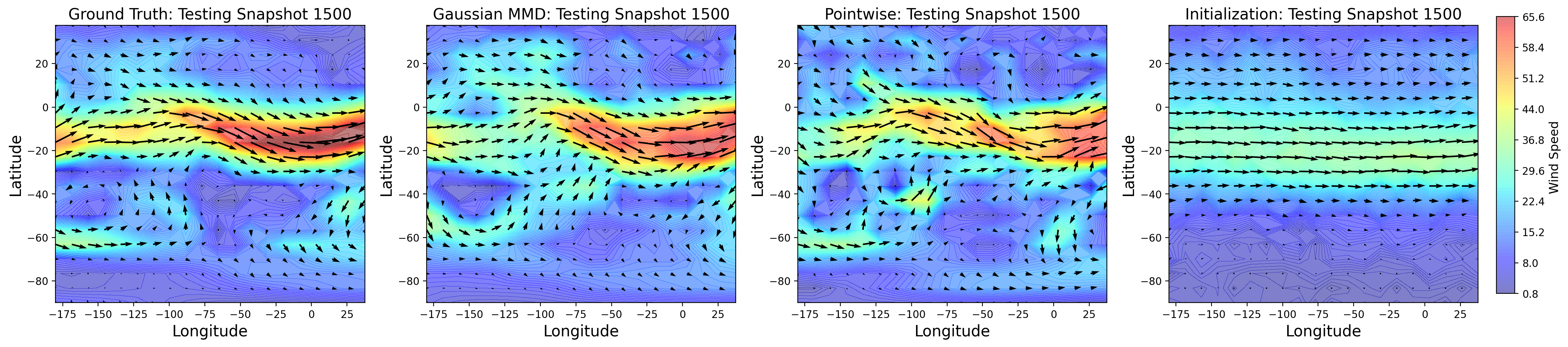}}\\

   \subfloat[Wind field reconstruction at testing week 3000]{ \includegraphics[width = \textwidth]{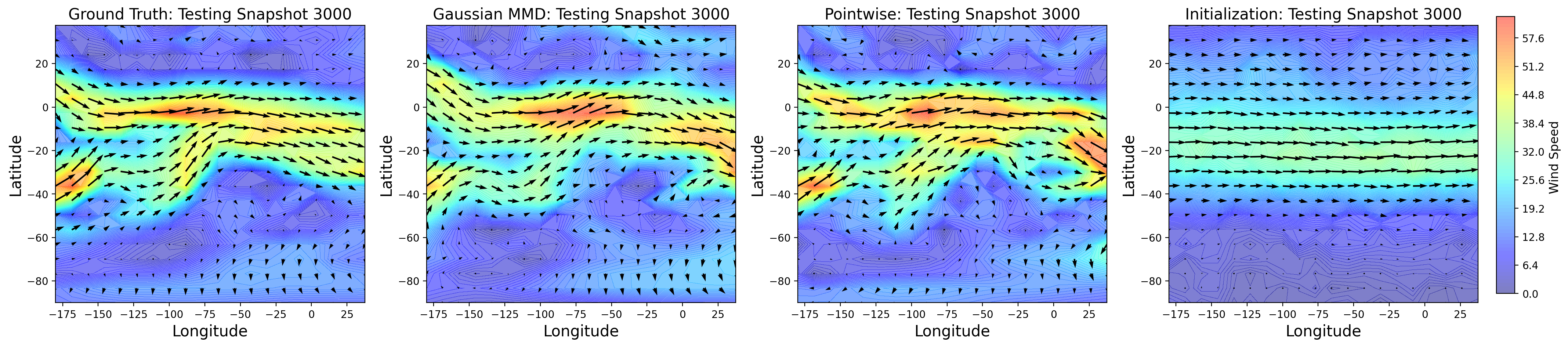}}\\
      \subfloat[Wind field reconstruction at testing week 4500]{ \includegraphics[width = \textwidth]{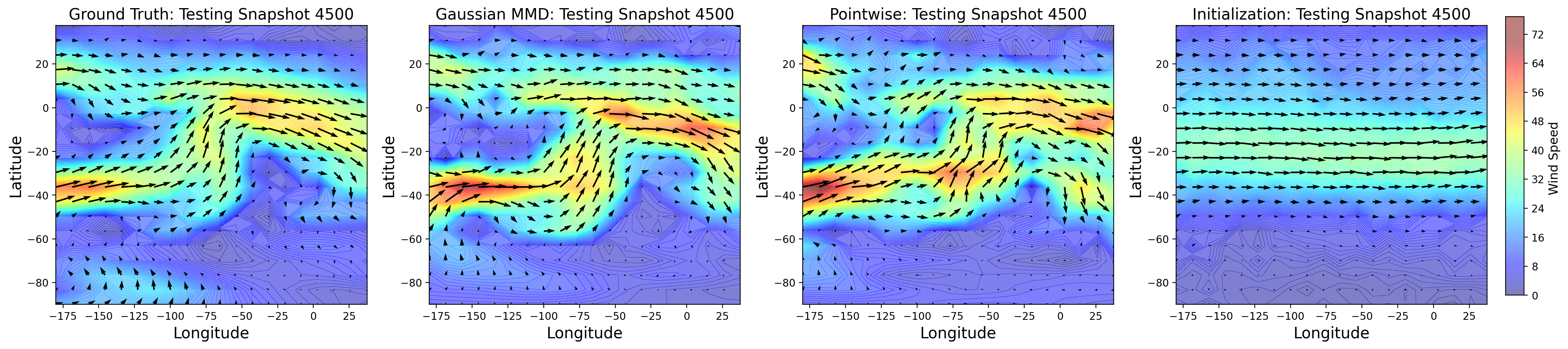}}
    \caption{Visual comparison of the pointwise and measure-based approaches to reconstructing the ERA5 wind dataset after 25000 training iterations. The initialization of the neural networks, shown in the final column, is exactly the same and the only difference among the columns is the loss function used during training.}
\label{fig:ERA5_RECON}
\end{figure}

\end{appendices}

\end{document}